\documentclass[francais]{smfart}

\usepackage{bull}
\usepackage[utf8x]{inputenc}
\usepackage[T1]{fontenc}

\usepackage{amsmath,amsfonts,amssymb,mathrsfs,latexsym,amscd,smfthm}
\usepackage{dsfont}
\usepackage{verbatim}
\usepackage{graphicx}
\usepackage{hyperref}
\usepackage[all,cmtip]{xy}
\usepackage{stmaryrd}
\newcommand{\BibTeX}{{\scshape Bib}\kern-.08em\TeX}

\newcommand\aess{\alpha_{\operatorname{ess}}}
\newcommand\Mor{\operatorname{Mor}}
\newcommand*{\dprime}{{\prime\prime}\mkern-1.2mu}

\newcommand{\CBY}{\mathcal{C}_{Q}^{\operatorname{b}}(Y_3)}
\theoremstyle{plain}
\newtheorem{theorem}{Théorème}[section]
\newtheorem{lemma}[theorem]{Lemme}
\newtheorem{proposition}[theorem]{Proposition}
\newtheorem{corollary}[theorem]{Corollaire}
\newtheorem{conjecture}[theorem]{Conjecture}
\newtheorem{fconjecture}[theorem]{Formule empirique}
\newtheorem{rconjecture}[theorem]{Répartition empirique}

\theoremstyle{definition}
\newtheorem{definition}[theorem]{Définition}

\newtheorem{example}[theorem]{Exemple}

\newtheorem{principle}[theorem]{Principe}
\newtheorem*{remark*}{Remarque}

\DeclareMathOperator{\pgcd}{\operatorname{pgcd}}

\DeclareMathOperator{\h}{\operatorname{H}}
\DeclareMathOperator{\e}{\textbf{e}}

\newcommand\RR{\mathbf{R}}

\newcommand\ZZ{\mathbf{Z}}
\newcommand\QQ{\mathbf{Q}}
\newcommand\NN{\mathbf{N}}
\newcommand\PP{\mathbf{P}}
\newcommand{\A}{\mathbf{A}}

\newcommand\rab{R_{r_1,r_2}}
\newcommand\cab{C_{a,b}}
\newcommand{\Spec}{\operatorname{Spec}}
\newcommand{\Ceff}{\mathcal{C}_{\operatorname{eff}}}

\newcommand{\MQ}{\mathcal{M}_Q^{\operatorname{res}}}

\title[Distribution locale sur une surface torique II]{\bf Approximation diophantienne \\ et distribution locale \\ sur une surface torique \uppercase\expandafter{\romannumeral2}}
\author{\textsc{Zhizhong Huang}}
\email{zhizhong.huang@yahoo.com}
\address{Institut für Algebra, Zahlentheorie und Diskrete Mathematik\\ \textit{Leibniz Universität Hannover}\\ \textit{Welfengarten 1, 30167 Hannover, Deutschland}.}

\date{}
	
\begin{document}
		\numberwithin{equation}{section}

\begin{abstract}
	Nous proposons une formule empirique pour le problème de distribution locale des points rationnels de hauteur bornée.
	Il s'agit d'une version locale du principe de Batyrev-Manin-Peyre.
	Nous la vérifions pour une surface torique, sur laquelle des courbes rationnelles cuspidales et des courbes rationnelles nodales toutes les deux contribuent aux meilleures approximations en dehors d'un fermé de Zariski. Nous démontrons qu'en enlevant une partie mince, il existe une mesure limite et une formule asymptotique pour le grossissement critique. 
\end{abstract}

\begin{altabstract}
	We propose an empirical formula for the problem of local distribution of rational points of bounded height.
	This is a local version of the Batyrev-Manin-Peyre principle. 
	We verify this for a toric surface, on which cuspidal rational curves and nodal rational curves all give the best approximations outside a Zariski closed subset. We prove the existence a limit measure as well as an asymptotic formula for the critical zoom by removing a thin set.
\end{altabstract}
\maketitle
\section{Introduction}
\subsection{Contexte et heuristique}
Concernant les variétés ayant beaucoup de points rationnels, une question naturelle est combien il y en a de hauteur bornée et comment ils sont distribués. 
Dans des années 1990, Batyrev et Manin ont conjecturé une formule asymptotique (cf. \cite[3.11 \& 3.12]{BatyrevManin}) qui donne une prédiction pour l'ordre de croissance du cardinal de l'ensemble des points rationnels de hauteur bornée. Peyre  (cf. \cite[Conjecture 2.3.1]{Peyre1}) a ensuite reformulé et raffiné leur conjecture sous une forme faisant intervenir des mesures, que nous appellerons \emph{distribution globale} et nous énonçons comme suit. Soit $X$ une \guillemotleft bonne\guillemotright{}\ variété sur $\QQ$ (cf. \cite[Notations 2.1]{Peyre2}). On note $X(\mathcal{A}_\QQ)^{\operatorname{Br}}$ l'ensemble des points adéliques de $X$ pour lesquels l'obstruction de Brauer-Manin à l'approximation faible est triviale. On associe une hauteur de Weil exponentielle $H$ au fibré anticanonique $\omega_{X}^{-1}$.
\begin{principle}[Batyrev-Manin-Peyre, \cite{Peyre2} Répartition empirique 5.3]\label{conj:bmp}
	Il existe un ouvert $U\subseteq X$ tel qu'en notant et $\kappa=\operatorname{rg}\operatorname{Pic}(X)$,
	\begin{equation}\label{eq:countmeasure}
	\delta_{U,B}=\sum_{\substack{P\in U(\QQ)\\H(P)\leqslant B}} \delta_P,
	\end{equation}
	on ait
	$$\frac{1}{B (\log B)^{\kappa-1}}\delta_{U,B} \longrightarrow \mu_{X}^{\operatorname{Br}},\quad B\to \infty$$
	au sens de convergence vague pour certaine mesure $\mu_{X}^{\operatorname{Br}}$ déduite d'un produit de mesures $\prod_{\nu\in \operatorname{Val}(\QQ)} \mu_{\nu}$ sur $X(\mathcal{A}_\QQ)$ par restriction à $X(\mathcal{A}_\QQ)^{\operatorname{Br}}$.
\end{principle}
Remarquons qu'en fait, la partie réelle $\mu_{\infty}$ est une mesure à densité continue relativement à la mesure de Lebesgue sur $X(\RR)$ (cf. \cite[2.2.1]{Peyre1}). Alors que la conjecture de Batyrev-Manin est vérifiée pour beaucoup de variétés presque de Fano \cite[Définition 3.1]{Peyre2}, voire singulières, ce principe sous forme de mesure est rarement abordée dans la littérature. Pourtant, les travaux de Batyrev et Tschinkel \cite{B-T1} \cite{B-T2} semblent être en faveur de cela pour les variétés toriques. 

Nous nous demandons dans quelle mesure une version plus forte de ce principe puisse être valide. À savoir, prenons un ouvert $D_Q(B)$ d'un point $Q\in X(\RR)$ pour la topologie réelle dont la taille dépend de $B$. Pourrions-nous espérer qu'il existe une mesure $\mu_{X,Q}$, définie localement sur le lieu réel de $X$ que nous préciserons dans la suite, telle que pour certain $\kappa^\prime\geqslant 1$,
\begin{equation}\label{eq:generalquestion}
\sum_{P\in D_Q(B)\cap X(\QQ):H(P)\leqslant B}\delta_P\sim \operatorname{Vol}(D_Q(B)) B (\log B)^{\kappa^\prime-1} \mu_{X,Q},\quad B\to \infty?
\end{equation}
Une motivation de ce problème est l'étude de la \emph{distribution locale} des points rationnels. Il fut considérée en premier par S. Pagelot \cite{pagelot}, où il a pris pour $D_Q(B)$ des boules de rayon $\asymp B^{-\frac{1}{r}}$ ($r$ sera appelé \emph{facteur de zoom} dans la suite) et où il a constaté, sur certaines surfaces toriques, des phénomènes variés pour la distribution locale des points rationnels autour d'un point fixé (décrite par la mesure $\mu_{X,Q}$ dans \eqref{eq:generalquestion}), même pour de différents $r$ d'une variété fixée. Tout cela n'est pas \emph{a priori} reflétée par la distribution globale (i.e. la mesure $\mu_X^{\operatorname{Br}}$ dans \eqref{eq:countmeasure}).

Pour établir l'existence de $\mu_{X,Q}$, il faut souvent retirer certaines sous-variétés \emph{localement accumulatrices} et bien choisir le facteur de zoom $r$. Les travaux de D. McKinnon et M. Roth (\cite{McKinnon2007} et \cite{McKinnon-Roth1}) concernant l'approximation diophantienne sur les variétés algébriques fournissent une constante $\alpha$ de nature arithmétique et géométrique, appelée \emph{constante d'approximation} (Définition \ref{def:alpha}). S. Pagelot définit dans \cite{pagelot} \emph{la constante essentielle} $\aess$ (Définition \ref{def:aess}) qui caractérise l'approximation diophantienne générique. Par définition, on a $\alpha\leqslant \aess$. En prenant le facteur de zoom $r$ entre $\alpha$ et $\aess$, la forme de $\mu_{X,Q}$ nous permet de récupérer plus d'informations qui sont \guillemotleft négligées\guillemotright{} dans la considération \eqref{eq:countmeasure}. Voir des explications et des illustrations dans \cite{huang2}. Nous espérons que les constantes $\alpha,\aess$ jouent un rôle tout comme les invariants de Fujita (les invariants $\alpha(L),t(L)$ dans \cite[2.1 \& 3.12]{BatyrevManin}) dans le programme de Batyrev-Manin-Peyre (Principe \ref{conj:bmp}) (cf. aussi le travail \cite{Lehmann-Tanimoto1} et les références dedans).

Dans \cite[(1.1)]{huang1} et \cite[(1.2)]{huang2} nous avons défini une famille de mesures qui capturent les points dans $D_Q(B)$. Maintenant nous continuons à proposer des formules asymptotiques prédisant l'ordre de grandeur dans le cas $r=\aess$. Nous désignons par
$A_1(X)$ le groupe de Chow des $1$-cycles modulo l'équivalence algébrique, et par 
$l(X)$ le rang du sous-groupe de $A_1(X)$ engendré par les classes des courbes rationnelles $C$ vérifiant $\alpha(Q,C)=\aess(Q)$. Tout au long de cet article, sauf si mentionné autrement, toutes les constantes d'approximation et tous les degrés seront calculés par rapport au fibré anti-canonique. Pour une partie $Y$ de $X(\QQ)$, nous notons $\{\delta_{Y,Q,B,r}\}_B$ la famille de mesures de zoom de facteur $r$ comptant les points rationnels sur $Y$ de hauteur $\leqslant B$ (cf. \S\ref{se:zoomoper}) définie sur l'espace tangent $(T_Q X)_\RR$. 
\begin{fconjecture}[version faible]\label{fc:weak}
	En dehors d'une partie mince $M$, nous avons que pour toute fonction $f$ continue à support compact définie sur $(T_Q X)_\RR$,
	\begin{equation}\label{eq:weak}
	\delta_{X\setminus M,Q,B,\aess}(f)=O_f\left(B^{1-\frac{\dim X}{\aess}} (\log B)^{l(X)-1}\right).
	\end{equation}
\end{fconjecture}
\begin{rconjecture}[version forte]\label{fc:strong}
	Si $\aess(Q)>\dim X$, alors en dehors d'une partie mince $M$, nous avons
	\begin{equation}\label{eq:conj}
	\frac{1}{B^{1-\frac{\dim X}{\aess}} (\log B)^{l(X)-1}}\delta_{X\setminus M,Q,B,\aess}\to \delta_{\aess},
	\end{equation}
	au sens de convergence vague pour une certaine mesure $\delta_{\aess}$ sur $(T_Q X)_\RR$.
\end{rconjecture}
L'étude des formules \ref{fc:weak} \& \ref{fc:strong} devrait ajouter une nouvelle évidence sur la connexion entre l'arithmétique des corps globaux et celle des corps de fonctions.
\subsection{Résultats principaux}
Suite aux travaux \cite{huang1} et \cite{huang2} qui démontrent la formule \ref{fc:strong} pour la surface $X_3$ et \ref{fc:weak} pour $Y_4$ respectivement, dans cet article on considère la surface torique $Y_3$ dont l'éventail est représenté au milieu de la Figure \ref{fg:y3} suivante. 
\begin{figure}[h]
	\centering
	\includegraphics[scale=0.65]{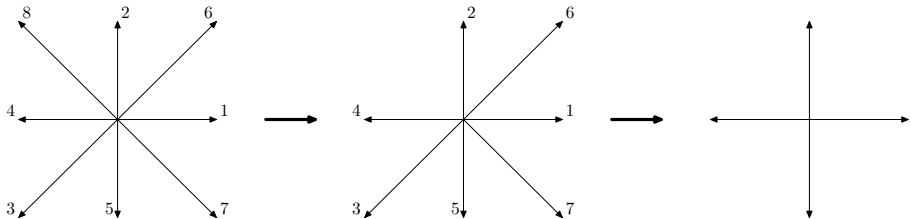}
	\caption{Les éventails de $Y_4$, de $Y_3$ et de $\PP^1\times\PP^1$}
	\label{fg:y3}
\end{figure}
Admettant $\PP^2$ et $\PP^1\times \PP^1$ comme modèles minimaux, elle est une surface de del Pezzo généralisée (cf. \cite{derenthal2014}) (c'est-à-dire la désingularisation minimale d'une surface de del Pezzo singulière de degré $5$).
Le résultat sur ces trois surfaces donne des évidences fortes sur le principe que le couple $(\alpha,\aess)$ et la façon dont elles sont calculées devraient caractériser l'accumulation locale des points rationnels.

Nous fixons tout au long ce travail le point central $Q=[1:1]\times [1:1]$. 
Notre premier résultat principal, concernant l'approximation de $Q$ par d'autres points rationnels sur $Y_3$, dit que des courbes nodales, qui couvrent une partie dense de $Y_3$ et font déjà des objets centraux d'étude pour la surface $Y_4$ \cite{huang2}, et des courbes cuspidales donnent en même temps les meilleurs approximants en dehors d'un fermé de Zariski (cf. \S\ref{se:cusp}, \S\ref{se:nodal}). Autrement dit, ces ceux types de courbes achèvent la constante $\alpha$.
\begin{theorem}[cf. Proposition \ref{po:lowerbound1}, Corollaire \ref{co:aess}]\label{th:main1}
	Nous avons 
	\begin{itemize}
		\item $\alpha(Q,Y_3)=2$. Elle s'obtient sur les trois courbes rationnelles lisses $l_i,(1\leqslant i\leqslant 3)$ de degré minimal passant par $Q$. Ces courbes sont localement accumulatrices;
		\item $\alpha_{\operatorname{ess}}(Q)=\frac{5}{2}$. Elle peut être calculée sur des courbes nodales et des courbes cuspidales passant par $Q$.
	\end{itemize}
\end{theorem}

Les courbes $l_i,1\leqslant i\leqslant 3$ sont en fait les (transformations strictes des) sections de bidegré $(1,1)$ joignant $Q$ et l'un des  $3$ points éclatés dans $\PP^1\times\PP^1$ (cf. \eqref{eq:curvesmalldegree}). Il s'en suit que la surface $Y_3$ et la partie $Y_3\setminus \cup_{i=1}^3 l_i$ vérifie la conjecture de D. McKinnon \cite[Conjecture 2.9]{McKinnon2007}. Bien que des courbes nodales et des courbes cuspidales aient la même valeur de constante d'approximation, le point $Q$ est approché de manière radicalement différente suivant elles (cf. \S\ref{se:compcuspnodal}), à cause de la façon dont $\alpha$ est calculée sur les courbes rationnelles (cf. Théorème \ref{th:singularities}). Le premier type correspond à l'approximation d'un nombre quadratique et le travail \cite{huang2} montre que le nombre des points rationnels entrant dans le zoom critique est faible et il n'existe pas de mesure décrivant la distribution locale (cf. Théorème \ref{th:huang2}). Alors que pour le deuxième on approche un point rationnel sur le corps de base $\QQ$. Dans ce cas l'ordre de grandeur est comparable avec la croissance de hauteur
et les points sont répartis asymptotiquement suivant une mesure limite (cf. Théorème \ref{th:Pagelot}).

Notre deuxième théorème principal confirme que la Répartition Empirique \ref{fc:strong} vaut pour $Y_3$. 
\begin{theorem}[cf. Théorèmes \ref{th:principalthm}, \ref{th:S2}, \ref{th:S3}]\label{th:main2}
	Soient $Z=\cup_{i=1}^3 l_i, U=Y_3\setminus Z$.
	\begin{enumerate}
		\item Pour $2=\alpha(Q,Y_3)\leqslant r<\frac{5}{2}$, nous avons $$\frac{1}{B^{1-\frac{1}{r}}}\delta_{Y_3,Q,B,r}\to \delta_r,$$ où $\delta_r$ est une mesure à support dans $Z$;
		\item Pour $r=\aess(Q)=\frac{5}{2}$, il existe une partie mince, qui est la réunion de $Z$ du type I et $M$ du type II, telle qu'en notant $V=Y_3(\QQ)\setminus (Z(\QQ)\cup M)$, nous ayons 
		$$\frac{1}{B^\frac{1}{5}}\delta_{V,Q,B,\frac{5}{2}}\longrightarrow \delta_{\frac{5}{2}},$$
		où $\delta_{\frac{5}{2}}$ est une mesure est absolument continue par rapport à la mesure de Lebesgue sur $\RR^2$.
	\end{enumerate}
\end{theorem}
C'est un fait empirique que la difficulté d'établir l'existence de la mesure limite augmente lorsque le degré de la surface baisse. Ceci est parallèle avec la conjecture de Batyrev-Manin sur la répartition globale \eqref{eq:countmeasure}. 
Même s'il n'est pas ardu d'établir une majoration uniforme d'ordre de grandeur $B^{\frac{1}{5}+\delta}$ (cf. Proposition \ref{po:uniformupperbound}) pour $r=\frac{5}{2}$, cependant, démontrer des formules asymptotiques ainsi que la convergence de mesures de zoom est un processus beaucoup plus délicat, et surtout on a dû surmonter la difficulté pour travailler avec le paramétrage donné par des courbes nodales dans \cite[\S5.2.2]{huang2}. Nous renvoyons au Théorème \ref{th:principalthm} pour une formule asymptotique précise comprenant la mesure $\delta_{\frac{5}{2}}$ et un terme d'erreur. Ce serait intéressant de pouvoir interpréter le facteur arithmétique qui apparaît dans le terme principal de façon géométrique comme dans le Principe \ref{conj:bmp}. Signalons que la densité de la mesure $\delta_{\frac{5}{2}}$ fait apparaître les trois courbes $l_i,1\leqslant i\leqslant 3$ qui sont localement accumulatrices. Ceci est analogue au résultat pour la surface $X_3$ (cf. \cite[Théorème 1.3]{huang1}). 

Les \emph{parties minces} (Définition \ref{def:thinsets}), dont la contribution a été considérée comme négligeable dans plupart de cas (cf. le théorème de S.D. Cohen, \cite[\S13 Theorem 1]{Serre}), se montrent parfois problématiques dans le programme de Batyrev-Manin (Principe \ref{conj:bmp}). Elles ont été reprises depuis le premier contre-exemple de V. V. Batyrev et Yu. Tschinkel \cite{batyrevtschinkel1996}. Mais c'est rare dans la littérature qu'on soit capable de contrôler le cardinal de la partie obtenue en retirant une partie mince. À la connaissance de l'auteur, les seules réussites jusqu'au présent sont le résultat du Rudulier \cite{LeRudulier} (cf. aussi \cite[\S8]{Peyre3}) où elle a considéré des schémas de Hilbert des points sur des surfaces, et celui de Browning et Heath-Brown \cite{Browning-HB} sur la variété bi-projective $\sum_{i=0}^{3}x_iy_i^2=0$ admettant une structure de fibration en quadratiques. Un fait intéressant pour la surface $Y_3$ est que la partie mince  de type II $M$ consiste en précisément les points sur des courbes cuspidales. Le Théorème \ref{th:main2} (2) dit que les parties minces ont aussi des influences non-négligeables pour le problème de distribution locale. Il nous fournit ainsi un autre exemple sur la gestion de ces ensembles.

Bien que la paire $(\alpha,\aess)$ soit un indicateur pour détecter les sous-variétés \emph{localement accumulatrices} (Définition \ref{def:aess}), on aurait besoin d'un critère vis-à-vis des points rationnels à retirer dans la démonstration des formules \ref{fc:weak} et \ref{fc:strong}. Dans \cite[\S4]{Peyre3}, s'inspirant des études de familles des courbes rationnelles sur les variétés, Peyre a introduit la notion de \emph{liberté} définie pour chaque point rationnel, analogue aux \emph{pentes} à la J.-B. Bost. Ceci a pour but d'améliorer le Principe \ref{conj:bmp}. Il semble raisonnable que cette notion puisse aussi s'insérer dans notre cadre (voir aussi \cite[\S7]{Peyre5} pour une autre discussion).

\subsection{Heuristique géométrique}
Pour soutenir les formules \ref{fc:weak} \& \ref{fc:strong}, nous tendons à donner, dans un travail à venir, une analogie géométrique 
s'inspirant celle du Principe \ref{conj:bmp}.
Il s'agit d'une sorte de \guillemotleft distribution\guillemotright{}\ des courbes rationnelles sur une variété algébrique. À la lumière d'un théorème de McKinnon et Roth (Théorème \ref{th:singularities}) et à l'aide de la théorie de déformation, en prenant en compte seulement des courbes ayant des multiplicités suffisamment grandes en $Q$, cette heuristique nous fournit une interprétation des ordres de grandeur sur $B$ et surtout sur $\log B$ compatible avec tous les résultats connus, comme l'avait déjà constaté par Pagelot \cite{pagelot} pour la surface $\PP^1\times\PP^1$ que le bon exposant sur $\log B$ ne coïncide pas avec celui dans le cas global. En réalité, pour $Y_3$ étudiée dans ce texte, toutes telles courbes réalisant la constante essentielle appartiennent aux classes $m[\omega_{Y_3}^{-1}],m\in\NN_{\geqslant 1}$, qui forment un sous-groupe de $A_1(Y_3)$ de rang $1$ et donc $l(Y_3)=1$ (ce qui suggère que la puissance sur $\log B$ devrait être $l(Y_3)-1=0$).

\subsection{Outils analytiques}
Expliquons les méthodes que l'on utilise, comment le zoom critique se traduit en des problèmes analytiques et d'où la partie mince intervient.
\subsubsection{Descente à la Legendre}
En considérant la famille de courbes nodales $\{\cab\}$ \eqref{eq:nodal} à paramètres $(a,b)\in\ZZ_{\operatorname{prem}}^2$ et en choisissant pour chaque $(a,b)$ un paramétrage $\PP^1\to \cab$ (dont les paramètres sont notés $(x,y)\in\ZZ_{\operatorname{prem}}^2$),
nous obtenons une application rationnelle $\PP^1\times\PP^1\dashrightarrow Y_3$
qui donne un paramétrage local des points rationnels sur $Y_3$ par les $\cab$.
L'idée clef ici est que plutôt que de dénombrer les points sur chaque courbe $\cab$, c'est-à-dire fixer le coupe $(a,b)$ en comptant $(x,y)$, ce qui était la méthode adoptée dans \cite{huang2}, on compte directement les couples $(a,b)\times (x,y)$ et il se trouve que les $(a,b)$ sont paramétrés par les $(x,y)$ quand on restreigne à des équations de Pell-Fermat avec certaine condition de divisibilité que l'on précise tout de suite. Ce passage nous permet d'éliminer les paramètres $(a,b)$. Il a des similarités avec la méthode de \guillemotleft descente\guillemotright{}\ qui remonte à Legendre et Dirichlet (cf. par exemple \cite{Lemmermeyer} pour un survol historique et les références dedans) et est récemment reprises par Fouvry et Jouve dans les travaux \cite{FouvryJouve3}, \cite{FouvryJouve1} et \cite{FouvryJouve2} pour traiter la solution fondamentale des équations de Pell-Fermat $x^2-Dy^2=1$.

Plus précisément, il s'avère qu'il faut se concentrer sur les solutions des équations du type comme par exemple, (cf. \eqref{eq:eqpellfermatgen} pour la forme la plus générale que l'on considère)
\begin{equation}\label{eq:introduction}
ax^2-by^2=b-a,
\end{equation}
avec $b>a>0, \pgcd(a,b)=1$.
En écrivant \eqref{eq:introduction} en 
$a(x^2+1)=b(y^2+1)$,
nous en déduisons que
$$a\mid y^2+1,\quad b\mid x^2+1,$$
puisque $\pgcd(a,b)=1$.
Nous arrivons donc au paramétrage 
$$x^2=bk-1,\quad y^2=ak-1,\quad k\in\NN,$$
et nous voyons que $\pgcd(x^2+1,y^2+1)=k=k(x,y)$ et $$a=\frac{y^2+1}{k(x,y)},\quad b=\frac{x^2+1}{k(x,y)}.$$
En rajoutant la condition supplémentaire 
\begin{equation}\label{eq:divisibility1}
b-a\mid x-y,
\end{equation}
qui est équivalente à 
\begin{equation}\label{eq:divtocongruence}
(x^2-y^2)/k(x,y)\mid x-y\Leftrightarrow x+y \mid k(x,y),
\end{equation}
la résolution de l'équation \eqref{eq:introduction} plus la condition \eqref{eq:divisibility1} finalement se transforme en
 $$x+y\mid k(x,y)\Leftrightarrow x^2\equiv -1\mod (x+y)\Leftrightarrow y^2\equiv-1\mod (x+y).$$
On est alors amené à un problème de distribution des racines de congruence polynomiale \emph{quadratique}, pour lequel nous appuyons sur un résultat de Hooley (cf. \cite[Theorem 3]{Hooley1}, \cite[Theorem 2]{Hooley2}).

La partie mince entre dans cette histoire en changeant le signe de l'équation \eqref{eq:introduction} comme suit.
$$ax^2-by^2=a-b.$$
Dans ce cas nous sommes amené au problème de congruence
$$x^2\equiv y^2\equiv 1 \mod (x+y).$$
Ceci n'est plus quadratique mais \emph{linéaire} puisque nous avons de manière équivalente,
$$(x+1)(x-1)\equiv 0\mod (x+y),$$
qui admet en général plus de solutions.

En effet, définissons pour $F(X)\in\ZZ[X]$ un polynôme entier,
$$\varrho_F(n)=\#\{1\leqslant m\leqslant n:n\mid F(m)\}.$$
Lorsque $\deg F(X)= 2$, il est classiquement connu d'après Ingham \cite{Ingham} et Erd\H{o}s \cite{Erdos} que
\begin{equation}\label{eq:asymp1}
\sum_{1\leqslant n\leqslant X}\varrho_F(n)\sim \begin{cases}
&c_F X; \text{ si } F(X) \text{ irréductible sur } \QQ;\\
&c_F X\log X; \text{ si } F(X) \text{ a deux racines rationnelles distinctes}.
\end{cases}
\end{equation}
\subsubsection{Répartition modulo $1$ des racines de congruence polynomiales}
Pour obtenir une mesure limite de dimension $2$, nous avons besoin en outre d'estimer des sommes plus raffinées que \eqref{eq:asymp1} suivantes qui ressemble à 
\begin{equation}\label{eq:sumtheta}
\sum_{\substack{1\leqslant l\leqslant m \leqslant X\\F(l)\equiv 0\mod  m}}\mathbf{1}_{]\theta_1,\theta_2]}\left(\frac{l}{m}\right).
\end{equation}
pour $0<\theta_1<\theta_2\leqslant 1$. 

Quand $F(X)$ est irréductible, Hooley \cite{Hooley1} \cite{Hooley2} démontré en premier que la série $(\frac{l}{m})$, formée par des couples $(l,m)\in\NN_{\geqslant 1}^2$ numérotés convenablement tels que $F(l)\equiv 0\mod m$, est \emph{uniformément répartie} dans l'intervalle $]0,1]$. Par conséquent, la somme \eqref{eq:sumtheta} ci-dessus se comporte comme $\sim c_F (\theta_2-\theta_1)X$. En pratique, le problème de comptage nous amène à une somme suivante plus générale que \eqref{eq:sumtheta}. Par un processus de \guillemotleft limite double\guillemotright{}, nous arrivons à démontrer que 
$$\sum_{\substack{1\leqslant l\leqslant m\\F(l)\equiv 0\mod m\\G(\frac{l}{m})m\leqslant X}}\mathbf{1}_{]\theta_1,\theta_2]}\left(\frac{l}{m}\right)\sim XC_F \int_{\theta_1}^{\theta_2}\frac{\operatorname{d}x}{G(x)},$$
où $G:\mathopen]0,1]\mathclose\to\RR_{>0}$ est une fonction continue bornée (cf. Proposition \ref{po:centralcouting}).

Le cas où $F(X)$ est réductible 
est récemment abordé par Dartyge et Martin \cite{Dartyge-Martin}. Leur résultat implique que les points rationnels correspondant à ce cas ne se répartissent pas assez aléatoirement et donc sont à retirer dans notre dénombrement. Plus de détails se trouvent dans \S\ref{se:notuniformdistr}.
\subsection{Organisation du texte} 
Nous commençons par rappeler la définition des constantes approximation et de l'opération de zoom dans \S\ref{se:approx}. 
Une discussion de la géométrie de $Y_3$ occupe de \S\ref{se:geometry}, où les équations de la famille de courbes nodales et celle de courbes cuspidales sont données en détails (\S\ref{se:curvesony3}). Nous démontrons le Théorème \ref{th:main1} et le Théorème \ref{th:main2} (1) dans \S\ref{se:appconstmaj} et nous donnons une majoration uniforme naïve  (Proposition \ref{po:uniformupperbound}) dont la preuve est relativement courte. Pour démontrer la Répartition empirique \ref{fc:strong}, nous avons besoin du paramétrage des points rationnels par des courbes nodales, donné dans \S\ref{se:parametrization}. La définition de partie mince est ensuite rappelée dans \S\ref{se:thinsets}. En particulier pour la surface $Y_3$ nous démontrons que l'ensemble des points se situant sur des courbes cuspidales forme une partie mince (\S\ref{se:thinsetiny3}). Le dénombrement se déroule dans \S\ref{se:counting}. Nous établions la convergence vague de la famille de mesures de zoom (\S\ref{se:criticalmeasure} et \S{\ref{se:otherregions}}) en dehors de la partie mince (\S\ref{se:countthethinpart}). 
Dans l'Appendice \ref{se:congHooley}, nous étudions, en suivant Erd\H{o}s et Hooley, des problèmes de congruence polynomiale dans les cas irréductibles (\S\ref{se:irred}) et scindés (\S\ref{se:splitcase}), avec une référence spéciale à l'équirépartition modulo $1$ des racines de congruence (\S\ref{se:uniformdistmodulo1} et \S\ref{se:notuniformdistr}). 

\subsection{Notations} 
$\mu(\cdot)$ désigne la fonction le Möbius, $\tau(\cdot)$ désigne la fonction donnant le nombre de diviseurs. La lettre $p$ est réservé aux nombres premiers. La condition $p^\nu \| n$ signifie que $p^\nu\mid n$ et $p^{\nu+1}\nmid n$. On désigne par $a=\square$ pour un nombre réel $a$ si $\sqrt{a}\in\QQ$.
Soit $E\subset\ZZ^n$, on note $E_{\operatorname{prem}}$ le sous-ensemble de $E$ définie par
$$E_{\operatorname{prem}}=\{(x_1,\cdots,x_n)\in E:\pgcd(x_i,1\leqslant i\leqslant n)=1\}.$$ Soient $b,c,m\in\ZZ$. On écrit désormais $b\equiv c[m]$ pour $b \equiv c\mod m$.
Définissons une fonction arithmétique $g:\NN_{\geqslant 1}\to\NN_{\geqslant 1}$ donnée par
\begin{equation}\label{eq:functiong}
g(n)=\prod_{p}p^{\lceil \frac{v_p(n)}{2}\rceil},\quad n\in\NN.
\end{equation}
Pour un domaine $I\subset \RR^n$, $\textbf{1}_I(\cdot)$ désigne sa fonction caractéristique. 
\tableofcontents
\section{Constante d'approximation et opération de grossissement}\label{se:approx}
\subsection{Constantes d'approximation et constantes essentielles}
	La notion de constantes d'approximation fut introduite en premier par D. McKinnon dans \cite{McKinnon2007}, comme étant une généralisation de \emph{mesure d'irrationalité} dans l'approximation diophantienne classique. Elle apparaît aussi dans le travail de pionnier de S. Pagelot \cite{pagelot} sur la distribution locale des points rationnels. Cette notion est ensuite reprise et réétudiée systématiquement par D. McKinnon et M. Roth dans \cite{McKinnon-Roth1}. Nous résumons brièvement quelques propriétés dont nous aurons besoin ici et nous renvoyons le lecteur à \cite[\S2]{McKinnon-Roth1}, \cite[\S1.1.3]{huang1} et \cite[\S2]{huang2} pour plus de détails.
	
	Soient $X$ une variété projective irréductible définie sur $\QQ$, $Q_0\in X(\bar{\QQ})$ un point algébrique réel fixé et $L$ un fibré en droites gros tel qu'il induise une application birationnelle de $X$ sur son image et bien définie sur un ouvert $U_0$ contenant $Q_0$. Nous y associons une hauteur de Weil exponentielle $H$ et nous fixons une distance projective archimédienne $d(\cdot,\cdot)$ (cf. \cite[\S2]{McKinnon-Roth1}).
\begin{definition}[McKinnon-Roth, cf. aussi \cite{huang2}, Définition 2.2]\label{def:alpha}
 Soit $V$ une partie constructible de $X$. \emph{La constante d'approximation} $\alpha(Q_0,V)$ est le supremum des $\delta>0$ tels que l'inégalité du type de Liouville suivante
 $$\exists C(\delta)>0,\quad \forall y\in (V\cap U_0)(\QQ)\setminus \{Q_0\},\quad d(Q_0,y)^\delta H(y)>C(\delta),$$
 soit valide. 
\end{definition}
D'après cette définition, la valeur de $\alpha(Q_0,V)$ est indépendante du choix de la distance et celui de la hauteur. Une plus petite valeur de la constante correspond à de meilleures approximations, contrairement à la notion classique de mesure d'irrationalité (puisque nous avons imposé l'exposant sur la distance). 

\begin{definition}[Pagelot]\label{def:aess}
	Soit $V$ une partie constructible de $X$. \emph{La constante essentielle} (par rapport à $V$) est 
	$$\alpha_{\operatorname{ess}}(Q_0,V)=\sup_{\substack{Y\subseteq V \text{partie constructible}\\ \text{dense pour la topologie de Zariski induite}}} \alpha(Q_0,Y).$$
	Par convention, $\aess(Q_0)=\aess(Q_0,X)$.
	$V$ est dite \emph{localement accumulatrice} si elle vérifie
	$$\alpha_{\operatorname{ess}}(Q_0,V)<\alpha_{\operatorname{ess}}(Q_0).$$
\end{definition}
La constante essentielle mesure donc l'approximation générique du point $Q_0$, et les variétés localement accumulatrices contiennent des points plus proches de $Q_0$.
Ces deux notions sont utiles pour interpréter le principe célèbre dans l'approximation diophantienne: toute approximation suffisamment \guillemotleft bonne\guillemotright{}\ (i.e. dont la constante d'approximation est plus petite que la constante essentielle) devrait être faite sur certaines sous-parties strictes de la variété.

Les courbes rationnelles jouent un rôle très important dans le comportement local des points rationnels (comme conjecturé par McKinnon \cite[Conjecture 2.7]{McKinnon2007}). Leur constantes d'approximation sont calculées de la façon suivante. 
\begin{theorem}[McKinnon-Roth, \cite{McKinnon-Roth1}, Theorem 2.16]\label{th:singularities}
Soient $C$ une courbe rationnelle définie sur $\QQ$ et $L$ un faisceau inversible ample sur $C$ et $Q_0\in C(\bar{\QQ})$. Soit $\phi:\PP^1\to C$ le morphisme de normalisation. Alors 
$$\alpha(Q_0,C)=\aess(Q_0,C)=\min_{P\in\phi^{-1}(Q_0)}\frac{d}{m_P r_P},$$
où
$d=\deg_C(L)$, $m_P$ est la multiplicité de la branche de $C$ passant par $Q_0$ correspondant à $P$ et 
\begin{equation*}
r(P)=\begin{cases}
0 \text{ si } k(P)\not\subset \RR;\\
1 \text{ si } k(P)=\QQ;\\
2 \text{ sinon},
\end{cases}
\end{equation*}
où par convention, $r_P=0$ signifie que $\frac{d}{m_Pr_P}=\infty$. 
\end{theorem}

\subsection{Opération de grossissement}\label{se:zoomoper}
On résume la définition de l'opération de zoom. Pour plus de détails, voir \cite[\S2.2]{huang2}.
\subsubsection{Énoncé du problème}

On note $\mathcal{C}^{\operatorname{b}}_{Q_0}X$ l'espace vectoriel des fonctions continues à support compact définie sur $(T_{Q_0} X)_\RR$ (l'espace tangent de $X$ en $Q_0$) à valeurs réelles. 
En ayant fixé un difféomorphisme local
$$\varrho:X(\RR)\dashrightarrow (T_{Q_0}X)_\RR,$$
pour $Y$ un sous-ensemble de $X(\QQ)$, la famille de mesures de zoom $\{\delta_{Y,Q_0,B,r}\}_B$ (de facteur $r$ par rapport à $Y$) est définie par
\begin{equation}\label{eq:zoommeasure}
\delta_{Y,Q_0,B,r}(f)=\sum_{x\in Y:H_L(x)\leqslant B} f(B^{\frac{1}{r}}\rho(x)),\quad \forall f\in \mathcal{C}^{\operatorname{b}}_{Q_0}X.
\end{equation}
Si $U$ est un sous-schéma de $X$, on écrit souvent $\delta_{U,Q_0,B,r}$ pour $\delta_{U(\QQ),Q_0,B,r}$.
Plus le facteur de zoom $r$ est grand, plus le zoom est faible (c'est-à-dire on compte plus de points).
 Si $r<\alpha(U,Q_0)$, alors on a $\delta_{U,Q_0,B,r}\to \delta_{Q_0}$ la mesure de Dirac concentrée sur $Q_0$ (\cite[Proposition 2.8]{huang2}).
 Le cas le plus intéressant est quand $U$ est une partie dense de $X$ tel que
 $\alpha(Q_0,U)=\aess(Q_0)<\infty$ et que l'on prend le facteur de zoom $r=\aess(Q_0)$, appelé \emph{critique}. Il semble exister souvent un \guillemotleft saut\guillemotright{}\ de dimension du support de mesures de zoom quand le facteur $r$ traverse $\aess(Q_0)$.
 On espère aussi que pour $r>\aess(Q_0)$, les points se distribuent de façon plus uniforme autour de $Q_0$.  Le résultat \cite[Théorème 1.2 (2)]{huang2} est un premier pas dans cette direction.
		
    Le problème \eqref{eq:generalquestion} maintenant se pose de la manière suivante:
	Existe-t-il $\beta=\beta(r),\gamma=\gamma(r)\geqslant 0$ et $\delta_r$ une mesure sur $(T_{Q_0}X)_\RR$ tels que
\begin{equation}\label{eq:qlocal1}
	\frac{\delta_{U,Q_0,B,r}}{B^\beta(\log B)^\gamma}\longrightarrow \delta_r,\quad B\to \infty
\end{equation}
	au sens de \emph{convergence vague}?

\begin{remark*}
	La formulation de la série \eqref{eq:zoommeasure} dépende \emph{a priori} de la hauteur et du difféomorphisme local choisis. La dernière est évidemment fonctorielle: un changement d'une carte locale devrait produire une autre mesure limite qui est la composée de l'ancienne avec cette application.
	Toutefois, dans tous les exemples connus, la fonction de densité de la mesure limite reste invariante et donc semble être une propriété intrinsèque et géométrique.
\end{remark*}
\subsection{Comparaison entre les courbes nodales et les courbes cuspidales}\label{se:compcuspnodal}
Comparons maintenant l'approximation sur les courbes cubiques ayant des singularités de types différents en le point à approcher.
Lorsque les pentes des tangentes en ce point sont irrationnelles, il s'agit essentiellement de l'approximation des nombres \emph{algébriques irrationnels} par les nombres rationnels. 
Lorsqu'elles sont rationnelles, on revient à l'approximation des nombres \emph{rationnels}.

Par meilleur illustrer cette différence, nous prendrons des cubiques affines d'équation simple se plongeant dans $\PP^2$.
On munit le fibré $\mathcal{O}(1)$ d'une hauteur de Weil absolue exponentielle $H$.
Choisissons sans perte de généralité le point $Q_0=[0:0:1]$ à approcher et
définissons la distance (pour $[u:v:w]\in\PP^2(\QQ)$)
$$d([u:v:w])=\max(\left|\frac{u}{w}\right|,\left|\frac{v}{w}\right|).$$

\noindent \textbf{Cas I.} Considérons la courbe nodale $\mathcal{C}:y^2=x^3+ax^2$ pour $a\in\QQ_{\geqslant 0}$.
Les tangentes en points $(0,0)$ ont les pentes $\pm\sqrt{a}$. 
Un paramétrage pour $\mathcal{C}$ est donné par 
$$k\longmapsto (k^2-a,k(k^2-a)).$$
On voit alors que la distance induite est $|k^2-a|$, qui est équivalente à 
$$\min(|k-\sqrt{a}|,|k+\sqrt{a}|).$$
Donc pour approcher le point $(0,0)$, il faut que $k$ approche $\pm \sqrt{a}$ (correspondant aux deux branches dans la normalisation).
La constante d'approximation est alors égale au degré (par rapport à $\mathcal{O}(1)$) de $\mathcal{C}$ divisé par $2$ si $a\neq\square$. Car, si l'on note $\phi:\mathcal{C}_0\to\mathcal{C}$ la normalisation, le sous-schéma $\phi^{-1}(Q_0)=\{\pm \sqrt{a}\}$ n'est pas défini sur $\QQ$. On a alors $r_P=2,\forall P$ dans le Théorème \ref{th:singularities}.
\begin{theorem}[\cite{huang2} Théorème 1.3]\label{th:huang2}
	Nous avons $\alpha(Q_0,\mathcal{C})=\alpha_{\operatorname{ess}}(Q_0,\mathcal{C})=\frac{3}{2}$, et
	$$\delta_{\mathcal{C},Q_0,\frac{3}{2},B}(\chi(\varepsilon))=O_\varepsilon(1).$$
\end{theorem}
\noindent \textbf{Cas II.} Considérons la courbe cuspidale $\mathcal{C}^\prime:y^2=x^3$ avec le point de rebroussement $Q_0=(0,0)$ ayant la tangente $y=0$.
Un paramétrage usuel pour cette courbe est
$$t \longmapsto (t^2,t^3).$$
D'où la distance induite se calcule comme $|t^2|$.
La constante d'approximation est alors divisée par $2$ car $m_P=2$ dans le Théorème \ref{th:singularities}.
\begin{theorem}[Pagelot \cite{pagelot}, cf. \cite{huang2} Théorème A.1]\label{th:Pagelot}
	Nous avons $\alpha(Q_0,\mathcal{C}^\prime)=\aess(Q_0,\mathcal{C}^\prime)=\frac{3}{2}$ et que
	$$\frac{1}{B^\frac{3}{4}}\delta_{\mathcal{C}^\prime,Q_0,\frac{3}{2},B}(\chi(\varepsilon))\to \delta_{\frac{3}{2}},$$
	une mesure à support dans $\mathcal{C}^\prime$.
\end{theorem}
On finit la discussion en observant que $\alpha(Q_0,\mathcal{C})=\alpha(Q_0,\mathcal{C}^\prime)$, mais $$\delta_{\mathcal{C},Q_0,\frac{3}{2},B}(\chi(\varepsilon))=o(\delta_{\mathcal{C}^\prime,Q_0,\frac{3}{2},B}(\chi(\varepsilon))).$$
La manière dont on calcule $\alpha$ peut entraîner de différents ordres de grandeur dans le zoom.
\section{Géométrie et courbes rationnelles sur $Y_3$}\label{se:geometry}
Nous considérons dans cet article la surface torique obtenue en éclatant $3$ des $4$ points invariants par l'action du tore de $\PP^1\times\PP^1$.
Nous désignerons par $Y_3$ cette surface puisque $Y_4$ a été réservée à la surface étudiée dans \cite{huang2}.
Sans perte de généralité, on peut supposer que $Y_3$ est l'éclatement en 
$$P_1=[1:0]\times [1:0],\quad P_2=[0:1]\times[1:0],\quad P_3=[1:0]\times[0:1].$$
On fixe désormais $$Q=[1:1]\times[1:1]$$ le point à approcher.
\subsection{Géométrie de $Y_3$}
Nous utilisons les coordonnées $[x:y]\times[s:t]$ de $\PP^1\times\PP^1$ pour les points différents de $P_i,1\leqslant i\leqslant 3$.
Rappelons l'éventail de $Y_3$ (Figure \ref{fg:y3}). 
Les courbes passant par $Q$ de degré minimal sont les (transformations strictes des) courbes rationnelles $l_i,1\leqslant i\leqslant 3$ définies par les équations
\begin{equation}\label{eq:curvesmalldegree}
y=x,\quad t=s,\quad yt=xs.
\end{equation}
On note $Z=\cup_{i=1}^3 l_i$ et $U=Y_3\setminus Z$. Il existe trois diviseurs exceptionnels $E_i~(1\leqslant i\leqslant 3)$ (par rapport au modèle minimal $\PP^1\times \PP^1$). 
La classe du diviseur anticanonique est $$[\omega_{Y_3}^{-1}]=\mathcal{O}(2,0)+\mathcal{O}(0,2)-[E_1]-[E_2]-[E_3].$$
Ses sections globales s'identifient à celles de $[\omega_{\PP^1\times\PP^1}^{-1}]$ qui s'annulent en les points éclatés. 
Une base $S$ est donnée par
\begin{equation}\label{eq:basisofglobalsec}
x^2 st,\quad y^2 st,\quad t^2 xy,\quad s^2 xy,\quad xyst,\quad y^2t^2.
\end{equation}
On identifie localement $T_Q Y_3$ sur la carte $(y\neq 0)\cap (t\neq 0)$ à l'espace affine $\RR^2$ via
\begin{equation}\label{eq:diffeomorphism}
\varrho:[x:y]\times [s:t]\longmapsto (w,z)=\left(\frac{x}{y}-1,\frac{s}{t}-1\right),
\end{equation}
sur lequel on utilise la distance pour le dénombrement:
\begin{equation}\label{eq:distance}
d((w,z))=\max(|w|,|z|).
\end{equation}
La courbe $yt=xs$ s'écrivant $wz+w+z=0$ (une hyperbole) sous ce difféomorphisme divise $\RR^2$ en deux régions 
\begin{equation}\label{eq:r1r2}
R_1=\{(w,z):wz+w+z>0\},\quad R_2=\{(w,z):wz+w+z<0\}.
\end{equation}
\subsection{Calcul de hauteur}\label{se:height}
Nous utilisons la hauteur de Weil absolue associée à $\omega_{Y_3}^{-1}$ définie par
$$H(P)=\frac{\max_{f\in S}(|f(P)|)}{\pgcd(f(P),f\in S)}, \quad P\in (Y_3\setminus \cup_{i=1}^3 E_i)(\QQ).$$
Pour un point $P=[x:y]\times[s:t]$ avec $$\pgcd(x,y)=\pgcd(s,t)=1,$$ on a
\begin{equation}\label{eq:pgcd}
\begin{split}
&\pgcd(x^2 st,y^2 st,t^2 xy,s^2 xy,xyst,y^2t^2)\\
=&\pgcd(\pgcd(x,s)\pgcd(x,t)\pgcd(y,s)\pgcd(y,t),y^2 t^2)\\
=&\pgcd(x,t)\pgcd(y,s)\pgcd(y,t).
\end{split}
\end{equation}

\subsection{Symétries}
De la structure de l'éventail de $Y_3$, nous constatons qu'il est symétrique par rapport à la droite engendrée par le $3$-ième rayon. Cette surface admet donc l'automorphisme relevant celui de $\PP^1\times\PP^1$ s'écrivant en coordonnées homogènes comme suit.
\begin{equation}\label{eq:symmetry}
\Phi:[x:y]\times[s:t]\longmapsto [s:t]\times[x:y].
\end{equation}
Il fixe le diviseur exceptionnel $E_1$, échange $E_2$ avec $E_3$
et préserve le diviseur anticanonique $\omega_{Y_3}^{-1}$ ainsi que la hauteur associée définie dans \S\ref{se:height}.
Dans les coordonnées $(w,z)$ \eqref{eq:diffeomorphism}, l'application $\Phi$ n'est rien d'autre que la permutation de coordonnées
$$(\varrho\circ \Phi\circ\varrho^{-1})(w,z)=(z,w).$$
Cette symétrie nous permet de ramener l'étude aux régions
\begin{equation}\label{eq:S1}
S_1=\{(w,z)\in\RR^2:w>0,z>w\},
\end{equation}
\begin{equation}\label{eq:S2}
S_2=\{(w,z)\in\RR^2:w<0,wz+w+z>0\},
\end{equation}
\begin{equation}\label{eq:S3}
S_3=\{(w,z)\in\RR^2:w>0,wz+w+z<0\},
\end{equation}
\begin{equation}\label{eq:S4}
S_4=\{(w,z)\in\RR^2:w<0,z<w\}.
\end{equation}
On a (rappelons \eqref{eq:r1r2})
$$S_1\cup S_2\subset R_1,\quad S_3\cup S_4\subset R_2.$$

\subsection{Courbes nodales et courbes cuspidales sur $Y_3$}\label{se:curvesony3}
Les équations à $3$ paramètres $c,d,e\in\ZZ^3_{\operatorname{prem}}$:
$$c(t-s)^2xy+d(y-x)^2st-eyt(y-x)(t-s)=0,$$
définissent des courbes comme étant des sections de $\omega_{Y_3}^{-1}$.
Pour une telle section irréductible ($\Leftrightarrow cde\neq 0$), un calcul de l'annulation de la matrice Hessienne donne que les pentes des tangentes en $Q$ se coïncident si et seulement si $e^2=4cd$, ce qui revient à la condition qu'il existe $(r_1,r_2)\in(\ZZ^2_{\neq 0})_{\operatorname{prem}}$ tel que
\begin{equation}\label{eq:coeffcusp}
c=r_1^2,\quad d=r_2^2,\quad e=2r_1r_2.
\end{equation}
\subsubsection{Famille de courbes cuspidales}\label{se:cusp}
Nous obtenons donc la famille de courbes cuspidales
\begin{equation}\label{eq:cusp}
\rab:r_1^2(t-s)^2xy+r_2^2(y-x)^2st-2r_1r_2yt(y-x)(t-s)=0.
\end{equation}
Dans les coordonnées $(w,z)$, elles s'écrivent
\begin{equation}\label{eq:cuspplane}
(r_1z-r_2w)^2+zw(r_1^2z+r_2^2 w)=0.
\end{equation}
Nous avons (rappelons l'automorphisme $\Phi$ \eqref{eq:symmetry}) $\Phi(R_{r_1,r_2})=R_{r_2,r_1}$ et la pente de la tangente en $Q=(0,0)$ est $\frac{r_2}{r_1}$. 
Les points rationnels sur la famille $\{R_{r_1,r_2}\}_{r_1,r_2\in\ZZ_{\operatorname{prem}}^2}$ de courbes forment un ensemble \emph{mince} de la région $R_2$. Nous en discutons plus loin dans \S\ref{se:thinsets}.

\subsubsection{Famille de courbes nodales}\label{se:nodal}
Lorsque les coefficients $c,d,e$ ne vérifient pas la condition \eqref{eq:coeffcusp}, nous obtenons en général une famille de courbes nodales à $3$ paramètres. La sous-famille qui sera utile est la suivante, également utilisée pour l'étude de la surface $Y_4$ dans \cite{huang2}:
\begin{equation}\label{eq:nodal}
\cab: axy(s-t)^2=bst(x-y)^2,\quad (a,b)\in\ZZ_{\operatorname{prem}}^2.
\end{equation}
Les tangentes en $Q=[1:1]\times[1:1]$ ont les pentes $\pm\sqrt{\frac{b}{a}}$. Elles sont donc irrationnelles si et seulement si $ab\neq \square$. 
\section{Déduction des constantes d'approximation et majoration naïve}\label{se:appconstmaj}
Le but de cette section est consacrée à la démonstration du Théorème \ref{th:main1}, puis sa conséquence, le Théorème \ref{th:main2} (1), et à celle d'une version faible du Théorème \ref{th:main2} (2). Elle s'agit d'une majoration naïve pour le zoom critique, qui est en faveur de l'heuristique de Batyrev-Manin (cf. \cite[\S2.2]{huang2}). 
\subsection{Bornes inférieures uniformes}
La proposition suivante correspond à la première partie du Théorème \ref{th:main1} (1). Elle se découle par une minoration directe. 
\begin{proposition}\label{po:lowerbound1}
	Nous avons
	$$\alpha(Q,Y_3)=2.$$
\end{proposition}
\begin{proof}
	Rappelons \eqref{eq:diffeomorphism} et \eqref{eq:distance}. Fixons $P=[x:y]\times[s:t] \neq Q$. Supposons sans perte de généralité que $s\neq t$. On a alors
	\begin{align*}
	\h_{\omega_{Y_3}^{-1}}(P) d(P)^2 &\geqslant\frac{|t^2 xy|}{\pgcd(x,t)\pgcd(y,s)\pgcd(y,t)} \left(\frac{s}{t}-1\right)^2\\
	&=\frac{|xy|}{\pgcd(x,t)\pgcd(y,s)\pgcd(y,t)}(s-t)^2\\
	&\geqslant 1.
	\end{align*}
	Cela montre que $\alpha(Q,Y_3)\geqslant 2$ par la Définition \ref{def:alpha}. Mais les courbes spéciales $l_i,1\leqslant i\leqslant 3$ de \eqref{eq:curvesmalldegree} vérifient $\alpha(Q,l_i)=2$. D'où 
	$\alpha(Q,Y_3)\leqslant \alpha(Q,l_i)=2$.
	Ceci clôt la démonstration.
\end{proof}
\subsection{Constante essentielle}
Ensuite on démontre une borne inférieure pour les points généraux en augmentant la puissance $2$ à $\frac{5}{2}$. On va l'utiliser pour obtenir la constante essentielle (Théorème \ref{th:main1} (2)).

\begin{proposition}\label{po:lowerbound2}
	Pour $P=[x:y]\times[s:t]$  n'appartenant pas aux trois courbes $l_i,1\leqslant i\leqslant 3$, i.e.
\begin{equation}\label{eq:tmp}
	x\neq y, \quad s\neq t,\quad xs\neq yt, 
\end{equation}
	tel que
	$d(P)\leqslant C$ avec $0<C<1$,
	il existe alors $D=D(C)>0$ tel que
	$$ \h_{\omega_{Y_3}^{-1}}(P) d(P)^\frac{5}{2}\geqslant D.$$
\end{proposition}
\begin{proof}
	Commençons par l'encadrement
	$$\left|\frac{x}{y}-1\right|\leqslant C\Rightarrow1-C\leqslant \left|\frac{x}{y}\right|\leqslant 1+C.$$
	Par conséquent, $$\left|\frac{y}{x}-1\right|=\left|\frac{x}{y}-1\right|\left|\frac{x}{y}\right|^{-1}\leqslant (1-C)^{-1}\left|\frac{x}{y}-1\right|.$$
	Et de même pour $\left|\frac{t}{s}-1\right|$.
	On a aussi
	\begin{align*}
		&\left|\frac{xs}{yt}-1\right|=\left|\frac{x}{y}\left(\frac{s}{t}-1\right)+\frac{x}{y}-1\right|\\
		&\leqslant (1+C)\max\left(\left|\frac{x}{y}-1\right|,\left|\frac{s}{t}-1\right|\right)=(1+C)d(P).
	\end{align*} 
	En minorant $d(P)$ de la manière suivante,
	$$d(P)^\frac{5}{2}\geqslant \left|\frac{x}{y}-1\right|\left|\frac{s}{t}-1\right|\left|\frac{xs}{yt}-1\right|^\frac{1}{2}(1+C)^{-\frac{1}{2}},$$
	et en rappelant l'hypothèse \eqref{eq:tmp}, nous avons alors
	\begin{align*}
	&\h_{\omega_{Y_3}^{-1}}(P) d(P)^\frac{5}{2}\\
	&\geqslant \frac{|xyst|}{\pgcd(x,t)\pgcd(y,t)\pgcd(y,s)}\left|\frac{y}{x}-1\right|\left|\frac{t}{s}-1\right|\left|\frac{xs}{yt}-1\right|^\frac{1}{2}\frac{(1-C)^{2}}{(1+C)^\frac{1}{2}}\\
	&=\left(\frac{|y|}{\pgcd(y,t)\pgcd(y,s)}\frac{|t|}{\pgcd(y,t)\pgcd(x,t)}\frac{|yt-xs|}{\pgcd(x,t)\pgcd(y,s)}\right)^\frac{1}{2}\\ &\quad \times|x-y||t-s|(1-C)^{2}(1+C)^{-\frac{1}{2}}\\
	&\geqslant (1-C)^2(1+C)^{-\frac{1}{2}}.
	\end{align*}
\end{proof}
\begin{corollary}\label{co:aess}
	Soit $U$ l'ouvert $Y_3\setminus (Z=\cup_{i=1}^3 l_i)$. Nous avons 
\begin{equation}\label{eq:tmp2}
	\alpha(U,Q)=\aess(Q)=\frac{5}{2}.
\end{equation}
	Les courbes \eqref{eq:curvesmalldegree} sont les variétés localement accumulatrices.
\end{corollary}
\begin{proof}
	Les courbes nodales $C_{a,b},ab\neq \square$ vérifient $\alpha(C_{a,b},Q)=\frac{5}{2}$ d'après le Théorème \ref{th:singularities} et leur réunion couvre l'ouvert dense $U$.
	D'où $$\aess(Q)\leqslant \alpha(Q,C_{a,b})=\frac{5}{2}.$$
	D'après la Proposition \ref{po:lowerbound2}, nous savons 
	$$\aess(Q)\geqslant \alpha(Q,U)\geqslant\frac{5}{2}.$$
	Cela démontre \eqref{eq:tmp2}.
	Alors que les courbes $l_i,1\leqslant i\leqslant 3$ de \eqref{eq:curvesmalldegree} sont lisse et de degré $2$ et donc vérifient $\alpha(Q,l_i)=2$.
\end{proof}
\begin{proof}[Démonstration du Théorème \ref{th:main2} (1)]
	On applique \cite[Proposition 2.7]{huang2} et l'on obtient que pour $2\leqslant r<\frac{5}{2}$, pour tout $f\in\CBY$ et $B\gg 1$, on a
	$$\delta_{Y_3,Q,B,r}(f)=\delta_{Z,Q,B,r}(f).$$
	La convergence de la suite $\{B^{-(1-\frac{1}{r})}\delta_{Z,Q,B,r}\}_B$ résulte du théorème de Pagelot (Théorème \ref{th:Pagelot}).
\end{proof}
\subsection{Majoration uniforme}
Sans utiliser les courbes nodales \eqref{eq:nodal}, nous démontrons le résultat faible suivant.
\begin{proposition}\label{po:uniformupperbound}
	Pour $\varepsilon>0$ fixé, on a que, pour tout $B$ suffisamment grand et pour tout $\delta>0$,
	$$\#\{P\in U(\QQ):\h_{\omega_{Y_3}^{-1}}(P)\leqslant B,d(P)\leqslant \varepsilon B^{-\frac{2}{5}}\}\ll_{\delta,\varepsilon }B^{\frac{1}{5}+\delta}.$$
\end{proposition}
\begin{proof}
	Introduisons les notations
	\begin{equation}\label{eq:ei}
	e_1=\pgcd(y,t),\quad e_2=\pgcd(x,t),\quad e_3=\pgcd(y,s);
	\end{equation}
	\begin{equation}\label{eq:fi}
	f_1=\frac{y}{e_1e_3},\quad f_2=\frac{t}{e_1e_2}, \quad f_3=\frac{yt-xs}{e_2e_3};
	\end{equation}
	\begin{equation}\label{eq:gi}
	g_1=\frac{x}{e_2},\quad g_2=\frac{s}{e_3}.
	\end{equation}
	Reprenons la démonstration de la proposition \ref{po:lowerbound2},
	$$1\gg_\varepsilon\h_{\omega_{Y_3}^{-1}}(P)d(P)^\frac{5}{2}\gg_\varepsilon |f_1f_2f_3|^\frac{1}{2} |x-y||t-s|.$$
	On voit que
\begin{equation}\label{eq:m0}
	|x-y|,\quad |t-s|,\quad f_1,\quad f_2,\quad f_3\quad \ll_\varepsilon 1.
\end{equation}
	
	On va ensuite déduire des encadrements pour les paramètres $g_1,g_2,e_1$,
	afin de démontrer qu'ayant les fixé, il n'y a qu'un nombre fini (dépendant de $\varepsilon$) de choix pour $e_2,e_3$.
	La condition de zoom 
	$$\left|\frac{x-y}{y}\right|\ll_\varepsilon B^{-\frac{2}{5}},\quad \left|\frac{s-t}{t}\right|\ll_\varepsilon B^{-\frac{2}{5}}$$
	 implique que 
	\begin{equation}\label{eq:m1}
	|t|\gg_\varepsilon B^{\frac{2}{5}},\quad |y|\gg_\varepsilon B^{\frac{2}{5}}.
	\end{equation}
	D'où l'on déduit  
	\begin{equation}\label{eq:m2}
	|s|\gg_\varepsilon B^{\frac{2}{5}},\quad |x|\gg_\varepsilon B^{\frac{2}{5}}.
	\end{equation}
	La condition que la hauteur soit bornée implique que $$|xse_1^2e_2e_3| \leqslant|xsf_1f_2e_1^2e_2e_3|=|xyst|\leqslant Be_1e_2e_3.$$
	D'où l'on déduit, grâce à \eqref{eq:m2},
\begin{equation}\label{eq:m3}
	e_1\ll_{\varepsilon}B^{\frac{1}{5}}.
\end{equation}
	D'après \eqref{eq:m0}, on a $$f_3=\frac{yt-xs}{e_2e_3}=f_1f_2e_1^2-g_1g_2\ll_{\varepsilon} 1.$$
	Donc une fois que $f_1, f_2$ et $e_1$ sont choisis, il n'y a qu'un nombre fini de valeurs possibles pour le produit $g_1g_2$. 
	Puisque 
\begin{equation}\label{eq:m4}
	\# \{(n_1,n_2)\in\NN_{\geqslant 1}^2:n_1n_2=n\}=\tau(n)=O(n^\delta)
\end{equation}
	pour tout $\delta>0$,
	on a donc démontré qu'ayant fixé $f_1,f_2,e_1$, le nombre des $(g_1,g_2)$ est $O_{\delta,\varepsilon}(B^\delta)$.
	
	Il nous reste à déterminer $e_2,e_3$. Fixons $f_1,f_2,e_1$ ainsi que $g_1$ et $g_2$. Comme $yt\neq xs$, on a $$f_1f_2e_1^2\neq g_1g_2\Leftrightarrow f_3\neq 0.$$
	Soit $C_2(\varepsilon)>0$ tel que $$\max(|y-x|,|t-s|)\leqslant C_2(\varepsilon).$$ 
	Le cardinal des $(e_2,e_3)$ possibles est majoré par (compte tenu des signes possibles de $x,y,s,t$)
	$$4\#\{(u,v)\in\ZZ_{\operatorname{prem}}^2:|ug_1-f_1e_1v|\leqslant C_2(\varepsilon)\text{ et }|vg_2-f_2e_1u|\leqslant C_2(\varepsilon)\},$$
	où on compte les points primitifs sur un réseau du déterminant $|f_3|$ dans un carré d'aire $4C_2(\varepsilon)^2$ dont l'intérieur contient l'origine. 
	On en conclut qu'il n'y a qu'un nombre fini de choix pour $e_2,e_3$ une fois que $f_1,f_2,e_1,g_1,g_2$ sont choisis (\cite[Lemma 2]{Heath-Brown2}). 
	La majoration qui fallait démontrer résulte de \eqref{eq:m0}, \eqref{eq:m3} et \eqref{eq:m4}.
\end{proof}
\section{Paramétrage par des courbes nodales}\label{se:parametrization}
On rappelle brièvement le paramétrage donné par la famille de courbes nodales \eqref{eq:nodal} utilisé de manière cruciale dans \cite{huang2}. 
L'équation de $\cab$ s'écrit sous les coordonnées $(w,z)$ comme
$$a(1+w)z^2=b(1+z)w^2.$$
Le paramétrage rationnel que l'on utilise est (cf. \cite[5.3.1]{huang2})
\begin{align}
\Psi:&[a:b]\times[u:v] \label{eq:paracab}\\ 
\longmapsto&[x:y]\times[s:t]= \left[\frac{bv(u-v)}{D_1d_2d_3}:\frac{u(bv-au)}{D_1d_2d_3}\right]\times \left[\frac{au(u-v)}{d_1D_2d_3}:\frac{v(bv-au)}{d_1D_2d_3}\right],\nonumber
\end{align}
où
\begin{equation}\label{eq:D1D2}
\begin{split}
d_1=\pgcd(u,b),\quad d_2=\pgcd(v,a),\quad d_3=\pgcd(u-v,b-a),\\
D_1=\pgcd(u^2,b),\quad D_2=\pgcd(v^2,a).
\end{split}
\end{equation}
Nous avons (cf. \cite[5.12]{huang2})
\begin{equation}\label{eq:pgcd2}
\pgcd(x,t)=\frac{v}{d_2},\quad \pgcd(y,s)=\frac{u}{d_1},\quad \pgcd(y,t)=\frac{bv-au}{d_1d_2d_3}.
\end{equation}
En le composant avec le difféomorphisme local $\varrho$, nous obtenons donc le paramétrage des points dans les régions $S_i,1\leqslant i\leqslant 4$ \eqref{eq:S1}-\eqref{eq:S4} comme suit:
\begin{equation}\label{eq:wz}
(\varrho\circ\Psi) \left((a,b)\times(u,v)\right)=(w,z)=\left(\frac{au^2-bv^2}{u(bv-au)},\frac{au^2-bv^2}{v(bv-au)}\right).
\end{equation}
En particulier nous voyons que $\frac{z}{w}=\frac{u}{v}$,
ce qui signifie que $\frac{u}{v}$ est la pente du point $(w,z)$.

Soient
\begin{equation}\label{eq:T1}
T_1=\left\{(a,b)\times(u,v)\in\NN_{\operatorname{prem}}^2\times\ZZ_{\operatorname{prem}}^2:b>a>0,u>v>0,\sqrt{\frac{b}{a}}<\frac{u}{v}<\frac{b}{a},\right\},
\end{equation}
\begin{equation}\label{eq:T2}
T_2=\left\{(a,b)\times(u,v)\in\NN_{\operatorname{prem}}^2\times\ZZ_{\operatorname{prem}}^2:b>a>0,u>-v>0,-\frac{u}{v}>\sqrt{\frac{b}{a}}\right\},
\end{equation}
\begin{equation}\label{eq:T3}
T_3=\left\{(a,b)\times(u,v)\in\NN_{\operatorname{prem}}^2\times\ZZ_{\operatorname{prem}}^2:b>a>0,u>-v>0,-\frac{u}{v}<\sqrt{\frac{b}{a}}\right\},
\end{equation}
\begin{equation}\label{eq:T4}
T_4=\left\{(a,b)\times(u,v)\in\NN_{\operatorname{prem}}^2\times\ZZ_{\operatorname{prem}}^2:b>a>0,u>v>0,\frac{u}{v}<\sqrt{\frac{b}{a}}\right\}.
\end{equation}
\begin{lemma}\label{le:au2-bv2}
	Pour $P=(a,b)\times (u,v)\in \cup_{i=1}^4 T_i$, nous avons $au^2-bv^2<0$ si et seulement si $P\in T_3\cup T_4$.
\end{lemma}
\begin{proof}
	Un calcul élémentaire nous donne: $$au^2-bv^2<0\Leftrightarrow -\sqrt{\frac{b}{a}}<\frac{u}{v}<\sqrt{\frac{b}{a}}\Leftrightarrow (a,b)\times (u,v)\in T_3\cup T_4.$$
\end{proof}
\begin{proposition}
	L'application $\varrho\circ\Psi$ donne une bijection entre l'ensemble $T_i$ et l'ensemble des $(w,z)\in\QQ^2$ satisfaisant à $\max(|w|,|z|)<1$ dans $S_i$ pour tout $1\leqslant i\leqslant 4$.
\end{proposition}
\begin{proof}
	Tout d'abord pour $(a,b)\times (u,v)\in (\NN_{\geqslant 1})^2_{\operatorname{prem}}\times (\NN_{\geqslant 1}\times\ZZ_{\neq 0})_{\operatorname{prem}}$ l'application $\varrho\circ\Psi$ \eqref{eq:wz} nous donne $(w,z)$ avec $\max(|w|,|z|)<1$.
	Réciproquement, étant donné $(w,z)\in \cup_{i=1}^4 S_i$ tel que $\max(|w|,|z|)<1$, les égalités
\begin{equation}\label{eq:inter1}
	\frac{b}{a}=\frac{z^2(1+w)}{w^2(1+z)}\quad(>0),\quad \frac{u}{v}=\frac{z}{w}
\end{equation}
	déterminent de façon unique les couples $(a,b)\in\NN_{\operatorname{prem}}^2$ et $(u,v)\in(\NN\times\ZZ)_{\operatorname{prem}}$. 
	De cette manière, nous avons que si $\max(|w|,|z|)<1$,
	\begin{align*}
		&\frac{b}{a}=\frac{z^2(1+w)}{w^2(1+z)}>1 \\
		\Leftrightarrow &(z-w)(z+w+zw)>0\\
		\Leftrightarrow &z>w \text{ et } z+w+zw>0 \quad \text{ou}\quad z<w \text{ et } z+w+zw<0,\\
		\Leftrightarrow &(w,z)\in \cup_{i=1}^4 S_i.
	\end{align*}
	De plus, 
	$$(w,z)\in S_1\Rightarrow \sqrt{\frac{1+w}{1+z}}\frac{z}{w}<\frac{z}{w}<\frac{z^2(1+w)}{w^2(1+z)}\Leftrightarrow \sqrt{\frac{b}{a}}<\frac{u}{v}<\frac{b}{a},$$
	$$(w,z)\in S_2\Rightarrow \sqrt{\frac{1+w}{1+z}}\left(-\frac{z}{w}\right)<-\frac{z}{w}\Leftrightarrow -\frac{u}{v}>\sqrt{\frac{b}{a}},$$
	$$(w,z)\in S_3\Rightarrow \sqrt{\frac{1+w}{1+z}}\left(-\frac{z}{w}\right)>-\frac{z}{w}>1\Leftrightarrow 1<-\frac{u}{v}<\sqrt{\frac{b}{a}},$$
	$$(w,z)\in S_4\Rightarrow \sqrt{\frac{1+w}{1+z}}\frac{z}{w}>\frac{z}{w}>1\Leftrightarrow \sqrt{\frac{b}{a}}>\frac{u}{v}>1.$$
	Nous avons donc établi $$(\varrho\circ\Psi)(T_i)=\{(w,z)\in\QQ^2:(w,z)\in S_i,\max(|w|,|z|)<1\}.$$
\end{proof}

La distance \eqref{eq:distance} se calcule sur $\cup_{i=1}^4 S_i$ par, d'après \eqref{eq:wz},
\begin{equation}\label{eq:distance2}
d((w,z))=\max\left(\left|\frac{au^2-bv^2}{u(bv-au)}\right|,\left|\frac{au^2-bv^2}{v(bv-au)}\right|\right)=\left|\frac{\frac{u^2}{v^2}-\frac{b}{a}}{\frac{b}{a}-\frac{u}{v}}\right|
\end{equation}
puisque nous avons supposé $|\frac{b}{a}|>1$ et donc $|\frac{u}{v}|=\left|\frac{z}{w}\right|>1$. 
\section{Parties minces}\label{se:thinsets}
\subsection{Définition}
 Le phénomène d'accumulation globale causé par des parties minces est déjà constaté dans plusieurs travaux, puisque le nombre des points de hauteur bornée qu'elles contiennent est parfois non-négligeable. Il s'avère qu'elles jouent aussi un rôle dans notre problème de comptage.
\begin{definition}[Definition 3.1.1 \cite{Serre2}]\label{def:thinsets}
	Soit $X$ une variété intègre sur $\QQ$. Soit $M$ un sous-ensemble de $X(\QQ)$ vérifiant qu'il existe une variété $V$ et un morphisme $f:V\to X$ tels que
	\begin{enumerate}
		\item $M\subseteq f(V(\QQ)) $;
		\item Le morphisme $f$ est génériquement fini et il n'admet pas de section rationnelle.
	\end{enumerate}
	Alors $M$ est dite du \emph{type I} si la fibre générique de $f$ est vide. 
	Il est dit du \emph{type II} si la variété $V$ est intègre et le morphisme $f$ est dominant.
	\emph{Une partie mince} est une réunion finie d'ensembles du type I ou II.
\end{definition}
Une partie mince du type I est donc contenue dans l'ensemble des points rationnels d'un fermé de Zariski. Alors que celle du type II est parfois dense dans $X$ pour la topologie de Zariski. 
\subsection{Ensemble mince dans $Y_3$}\label{se:thinsetiny3}
Considérons l'ensemble
\begin{equation}\label{eq:thinsetM}
M=\{(w,z)\in\QQ^2:-w-z-wz=\square\neq 0\}\subset R_2.
\end{equation}
C'est un ensemble se trouvant dans la partie au-dessous de l'hyperbole $wz+w+z=0$ et dense dans $S_2$ pour la topologie réelle.

D'après le paramétrage $\Psi$ \eqref{eq:paracab}, 
$$-w-z-wz=\square\Leftrightarrow  -\frac{(au^2-bv^2)(b-a)}{(bv-au)^2}=\square\Leftrightarrow \left(\frac{u^2}{v^2}-\frac{b}{a}\right)\left(1-\frac{b}{a}\right)=\square,$$
Soit $V$ la sous-variété de $\A^3_{x_1,x_2,x_3}$ d'équation
$$x_3^2=\left(x_2^2-x_1\right)\left(1-x_1\right).$$
Pour $y\in\QQ$, soit $(y(z_1),y(z_2))\in (\NN\times\ZZ_{\neq 0})_{\operatorname{prem}}$ défini par $y=\frac{y(z_1)}{y(z_2)}$. Considérons l'application rationnelle
$\pi:V \dashrightarrow Y_3$
donnée par
$$\pi\left(x_1,x_2,x_3\right)=\Psi((x_1(z_1),x_1(z_2))\times (x_2(z_1),x_2(z_2))).$$
Elle est donc génériquement de degré $2$. 
Alors nous avons que $M\subset \operatorname{Im}(\pi)(\QQ)$.
D'où $M$ est une partie mince dans $Y_3$.

\begin{proposition}\label{po:thincusp}
	Un point $P=(w,z)\in\QQ^2$ est dans $M$ si et seulement s'il existe une courbe cuspidale $R_{r_1,r_2}$ \eqref{eq:cuspplane} passant par $P$.
\end{proposition}
\begin{proof}
	Une des courbes $R_{r_1,r_2}$, vue comme une équation en $\lambda=\frac{b}{a}$, s'écrit comme
\begin{equation}\label{eq:zwlambda}
	w^2(1+z)\lambda^2- 2zw\lambda+z^2(1+w)=0,
\end{equation}
	avec le discriminant $$\triangle=4w^2z^2(1-(1+w)(1+z))=-4w^2z^2(wz+w+z).$$
	Donc un point $(w,z)\in\QQ^2$ est sur $R_{r_1,r_2}$ seulement si 
	\begin{equation}\label{eq:square}
	\triangle=\square \Leftrightarrow -(wz+w+z)=\square.
	\end{equation}
	Réciproquement, si un point $(w,z)\in M$, alors l'équation \eqref{eq:zwlambda} admet deux solutions rationnelles distinctes, qui correspondent à deux courbes cuspidales passant par $M$.
\end{proof}

\section{Dénombrement d'approximants}\label{se:counting}
Cette partie technique est consacrée à la démonstration du Théorème \ref{th:main2} (2). Après une étape préparatoire de finitude et de réduction (\S\ref{se:finiteness} \& \S\ref{se:reduction}), on établit le résultat principal -- le Théorème \ref{th:principalthm}, dans \S\ref{se:criticalmeasure} et \S\ref{se:otherregions} et le Théorème \ref{th:main2} (2) en résulte. À la fin dans \S\ref{se:countthethinpart} on finit par une discussion courte sur le nombre des points dans l'ensemble mince.
\subsection{Finitude des paramètres}\label{se:finiteness}
Le lemme suivant détermine les paramètres possibles qui sont de nombre fini quand on regarde le zoom dans un voisinage borné fixé. Rappelons que $U=Y_3\setminus\cup_{i=1}^3 l_i$ \eqref{eq:curvesmalldegree} et les notations \eqref{eq:D1D2}.
\begin{lemma}\label{le:finitenessofpara}
Soient $\varepsilon>0$. Nous avons que pour tout $P=[a:b]\times [u:v]\in U(\QQ)$ tel que $d(P)^\frac{5}{2}\h_{\omega_{Y_3}^{-1}}(P)\leqslant \varepsilon$,
$$D_1,\quad D_2,\quad \frac{|b-a|}{d_3},\quad \frac{|au^2-bv^2|}{D_1D_2d_3}\ll_\varepsilon 1.$$
\end{lemma}
\begin{proof}
Reprenons les notations \eqref{eq:ei}, \eqref{eq:fi} et rappelons encore la démonstration de la Proposition \ref{po:uniformupperbound}. Compte tenu du paramétrage dans \S\ref{se:parametrization}, nous voyons que les quantités
$$|x-y|=\frac{|au^2-bv^2|}{D_1d_2d_3},\quad |t-s|=\frac{|au^2-bv^2|}{D_2d_1d_3},$$
$$|f_1|=\frac{|y|}{e_1e_3}=\frac{d_1^2}{D_1},\quad |f_2|=\frac{|t|}{e_1e_2}=\frac{d_2^2}{D_2},$$
$$|f_3|=\frac{|yt-xs|}{e_2e_3}=\frac{|(b-a)(au^2-bv^2)|}{D_1D_2d_3^2}.$$
sont finies dans tout voisinage borné grâce à \eqref{eq:m0}. Il en est de même pour les quantités 
$$\frac{|b-a|}{d_3},\quad \frac{|au^2-bv^2|}{D_1D_2d_3}.$$
De plus, l'égalité
$$|x-y|\times |f_2|=\frac{|au^2-bv^2|}{D_1d_2d_3}\times \frac{d_2^2}{D_2}=\frac{|au^2-bv^2|}{D_1D_2d_3}\times d_2,$$
et la finitude ci-dessus impliquent que $d_2$ ainsi que $D_2$ sont finies grâce à la relation $d_2\leqslant D_2\leqslant d_2^2$. 
De même pour $d_1$ et $D_1$ en considérant $|t-s|\times |f_1|$.
\end{proof}
\subsection{Réduction aux problèmes de congruences}\label{se:reduction}
Pour $C_3\in\NN_{\geqslant 1},D\in\ZZ_{\neq 0}$ avec $\pgcd(C_3,D)=1$, on considère l'équation suivante en $(a,b)\times(u,v)\in\ZZ_{\operatorname{prem}}^2\times\ZZ_{\operatorname{prem}}^2$:
\begin{equation}\label{eq:eqpellfermatgen}
\mathcal{E}_{C_3,D}:C_3(au^2-bv^2)=D(b-a).
\end{equation} 
\begin{lemma}\label{le:pgcd}
	Pour tout $(a,b)\times (u,v)\in(\NN_{\geqslant 1}^2)_{\operatorname{prem}}\times(\ZZ_{\neq 0}^2)_{\operatorname{prem}}$ vérifiant $C_3\mid b-a$ et \eqref{eq:eqpellfermatgen}, nous avons $$\pgcd(u^2,b)\pgcd(v^2,a)\mid D,\quad \pgcd(u-v,b-a)\mid \frac{b-a}{C_3}.$$
\end{lemma}
\begin{proof}
	Puisque $$au^2-bv^2=a(u^2-v^2)-(b-a)v^2=b(u^2-v^2)-u^2(b-a),$$ nous avons
	\begin{equation}\label{eq:div0}
	\pgcd(u^2,b)\pgcd(v^2,a)\pgcd(u-v,b-a)\mid au^2-bv^2=\frac{b-a}{C_3}\times D.
	\end{equation}
	Grâce à \eqref{eq:div0}, la première divisibilité est évidente puisque ces deux termes sont premiers avec $b-a$. Pour voir la deuxième, il suffit simplement d'observer que pour tout nombre premier $p\mid \pgcd(u-v,b-a)$, écrivons $p^k\| \pgcd(u-v,b-a)$ pour certain $k\geqslant 1$. Si $p\mid D$, comme $\pgcd(C_3,D)=1$ et $\pgcd(u-v,b-a)\mid b-a$, on obtient 
	\begin{equation}\label{eq:divisibility}
	p^k\mid \frac{b-a}{C_3}.
	\end{equation}
	Pour les $p\nmid D$, \eqref{eq:divisibility} est automatique, grâce encore à \eqref{eq:div0}. 
\end{proof}
Le lemme suivant est une généralisation de la transformation faite dans l'introduction du texte pour l'équation \eqref{eq:introduction}.
\begin{lemma}\label{le:parametrizationofpell}
	Fixons $(u,v)\in\ZZ_{\operatorname{prem}}^2$ tel que \begin{equation}\label{eq:uv}
	\min(u^2,v^2)>-\frac{D}{C_3},
	\end{equation}
	alors tout $(a,b)\in(\NN_{\geqslant 1}^2)_{\operatorname{prem}}$ vérifiant \eqref{eq:eqpellfermatgen} s'écrit
\begin{equation}\label{eq:ab}
	(a,b)=\left(\frac{C_3 v^2+D}{k(u,v)},\frac{C_3 u^2+D}{k(u,v)}\right), 
\end{equation}
où on pose pour $(x,y)\in\ZZ^2$,
$$k(x,y)=\pgcd(C_3x^2+D,C_3y^2+D).$$
\end{lemma}
\begin{proof}
	Étant donné un $(a,b)\in(\NN_{\geqslant 1}^2)_{\operatorname{prem}}$ vérifiant $\mathcal{E}_{C_3,D}$ \eqref{eq:eqpellfermatgen}, nous avons
	$(C_3u^2+D)a=(C_3v^2+D)b$,
	et un raisonnement identique comme pour \eqref{eq:introduction} donne bien \eqref{eq:ab}.
\end{proof}


L'observation suivante est l'étape cruciale pour réduire le comptage aux problèmes de congruences. Elle joue un rôle tout comme le passage de \eqref{eq:introduction} à \eqref{eq:divtocongruence} en supposant la condition \eqref{eq:divisibility1}.
Introduisons la notation remplaçant désormais $d_3$:
\begin{equation}\label{eq:c3}
c_3=\frac{b-a}{d_3}=\frac{b-a}{\pgcd(u-v,b-a)}.
\end{equation}

\begin{proposition}\label{le:keytranslation}
	Fixons $q\in\NN_{\geqslant 1}$. Pour tout $(a,b)\times (u,v)\in(\NN_{\geqslant 1}^2)_{\operatorname{prem}}\times(\ZZ_{\neq 0}^2)_{\operatorname{prem}}$ vérifiant \eqref{eq:uv} et \eqref{eq:eqpellfermatgen}, la condition
	\begin{equation}\label{eq:quadraticcong}
	u+v\mid qk(u,v).
	\end{equation}
	équivaut à l'assertion suivante:
\begin{equation}\label{eq:W}
	c_3\mid C_3q.
\end{equation}
\end{proposition}
\begin{proof}
	Grâce au Lemme \ref{le:pgcd}, nous avons
	$$\pgcd(u-v,b-a)=\frac{b-a}{c_3}\mid \frac{b-a}{C_3},$$
	et donc $C_3\mid c_3$.
	Adaptons les notations introduites dans le Lemme \ref{le:parametrizationofpell}.
	Soit $q^\prime\mid q$ tel que $c_3=q^\prime C_3$. Nous avons
	\begin{align*}
	\frac{b-a}{C_3q^\prime}=\frac{u^2-v^2}{k(u,v)q^\prime}&=\pgcd(u-v,b-a)\\
	&=\pgcd\left(u-v,\frac{C_3(u^2-v^2)}{k(u,v)}\right)\\
	&=\pgcd\left(\frac{ k(u,v)}{u+v}\times \frac{u^2-v^2}{k(u,v)},C_3\times \frac{u^2-v^2}{k(u,v)}\right).
	\end{align*}
	Alors nous en déduisons que $u+v\mid q^\prime k(u,v)$.
	En particulier si $q^\prime\mid q$, on a \emph{a fortiori} $u+v\mid qk(u,v)$.
	
	Supposons maintenant \eqref{eq:quadraticcong} et écrivons $$qk(u,v)=\lambda(u+v)$$ pour $\lambda\in\NN_{\geqslant 1}$. Définissons
	$$q^\prime=\frac{q}{\pgcd(q,\lambda)},\quad \lambda^\prime=\frac{\lambda}{\pgcd(q,\lambda)}.$$ 
	Nous avons $q^\prime k(u,v)=\lambda^\prime (u+v), \pgcd(\lambda^\prime,q^\prime)=1$.
	D'où nous déduisons que $q^\prime\mid u+v$ et $\lambda^\prime\mid k(u,v)$.
	Cela nous permet d'écrire
	$$\frac{b-a}{C_3}=\frac{u^2-v^2}{k(u,v)}=\frac{q^\prime(u-v)}{\lambda^\prime},$$
	et nous avons donc $\lambda^\prime\mid u-v$.
	Alors
	\begin{align*}
	\pgcd(u-v,b-a)&=\pgcd\left(u-v,\frac{C_3(u^2-v^2)}{k(u,v)}\right)\\
	&=\pgcd\left(u-v,\frac{C_3q^\prime(u-v)}{\lambda^\prime}\right)\\
	&=\frac{u-v}{\lambda^\prime}\pgcd(\lambda^\prime,C_3q^\prime)\\
	&=\frac{b-a}{C_3q^\prime},
	\end{align*}
	puisque $\pgcd(k(u,v),C_3)=1,\lambda^\prime\mid k(u,v)$. 
	D'où $c_3=C_3q^\prime$ avec $q^\prime\mid q$.
\end{proof}

La prochaine proposition révèle que les courbes cuspidales interviennent dans le dénombrement en changeant le signe de \eqref{eq:eqpellfermatgen}. 
\begin{proposition}\label{le:splitpoly}
	Rappelons l'ensemble mince $M$ \eqref{eq:thinsetM} et notons $A_{C_3,D}$ l'ensemble des $(a,b)\times (u,v)\in(\NN_{\geqslant 1}^2)_{\operatorname{prem}}\times(\ZZ_{\neq 0}^2)_{\operatorname{prem}}$ vérifiant \eqref{eq:uv} et \eqref{eq:eqpellfermatgen}.
	Nous avons
	$(\varrho\circ \Psi)(A_{C_3,D})\cap M\neq\varnothing$
	si et seulement si $C_3,-D=\square$.
	En particulier dans ce cas le polynomial $C_3X^2+D$ est réductible dans $\ZZ[X]$.
	Par conséquent, $\varrho\circ \Psi$ restreinte à $$\bigsqcup_{\substack{(C_3,D)\in\NN_{\geqslant 1}\times \ZZ_{< 0}\\\pgcd(C_3,D)=1\\C_3,-D=\square}}A_{C_3,D}\cap (T_3\cup T_4)\to (T_Q Y_3)_\RR$$ donne un paramétrage de l'ensemble mince $$\{(w,z)\in M\cap(S_3\cup S_4):\max(|w|,|z|)<1\}.$$ 
\end{proposition}
	D'après la Proposition \ref{po:thincusp}, l'ensemble $M$ est compris des points sur les courbes cuspidales $R_{r_1,r_2}$. Pour la topologie réelle, une telle $R_{r_1,r_2}$ a plusieurs composantes connexes. Le résultat ci-dessus implique que celle qui passe par $Q$ vivent seulement dans la région $(S_3\cup S_4)\cup \Phi(S_3\cup S_4)$ ($\Phi$ est l'automorphisme \eqref{eq:symmetry}).
\begin{proof}
	Prenons $(w,z)\in\QQ^2\cap (\cup_{i=1}^4 S_i)$.
	Rappelons le paramétrage pour les coordonnées $(w,z)$ \eqref{eq:wz}.
	Nous calculons
	\begin{align*}
	wz+w+z&=\frac{(au^2-bv^2)^2}{uv(bv-au)^2}+\frac{au^2-bv^2}{u(bv-au)}+\frac{au^2-bv^2}{v(bv-au)}\\
	&=\frac{(au^2-bv^2)(b-a)}{(bv-au)^2}.
	\end{align*}
	Par la définition de l'ensemble $M$, $\varrho\circ \Psi(\cup_{i=1}^4 T_i)\cap M\neq \varnothing$ si et seulement si 
	$$-(au^2-bv^2)(b-a)=\square.$$
	En particulier nous voyons que $au^2-bv^2<0$ puisque $b>a$. Nous déduisons du Lemme \ref{le:au2-bv2} que l'image d'un tel point est dans $S_3\cup S_4$. De l'équation $\mathcal{E}_{C_3,D}$ \eqref{eq:eqpellfermatgen}, la condition ci-dessus est équivalente à 
	$$-(b-a)^2\frac{D}{C_3}=\square\Leftrightarrow -\frac{D}{C_3}=\square.$$
	Le résultat en découle puisque nous avons imposé que $D,C_3$ soient premiers entre eux et que $C_3>0$.
	Le dernier énoncé découle de la Proposition \ref{po:thincusp}.
\end{proof}

\subsection{Région $S_1$}\label{se:criticalmeasure}
Grâce à la similitude du calcul, nous nous bornons alors pour la suite de cette sous-section à la région $S_1$ \eqref{eq:S1}. Comme expliqué précédemment, cette région n'intersecte pas l'ensemble mince $M$.
Pour $\varepsilon_1>\varepsilon_2>0,\theta_1>\theta_2>1$ fixés, on désigne par $\underline{\varepsilon},\underline{\theta}$ pour ces paramètres et l'on considère
\begin{equation}\label{eq:chi}
R(\underline{\varepsilon},\underline{\theta})=\{(w,z)\in S_1: \varepsilon_2<z\leqslant\varepsilon_1,\theta_2w<z\leqslant \theta_1w\}, 
\end{equation}
une région trapézoïdale et $\chi_{\underline{\varepsilon},\underline{\theta}}(\cdot)=\textbf{1}_{R(\underline{\varepsilon},\underline{\theta})}(\cdot)$ la fonction caractéristique. Elle sert d'une fonction \guillemotleft test\guillemotright{}. Pour déduire le Théorème \ref{th:main2} (2), il suffit d'établir la convergence de la suite $\{\delta_{\varrho^{-1}(S_1)\cap U,Q,B,\frac{5}{2}}(\chi_{\underline{\varepsilon},\underline{\theta}})\}_B$, car toute fonction continue à support compact est la limite uniforme d'une suite de fonctions caractéristique de la forme $\chi_{\underline{\varepsilon},\underline{\theta}}$.

\begin{theorem}\label{th:principalthm}
	On a
	\begin{align*}
		&\delta_{\varrho^{-1}(S_1)\cap U,Q,B,\frac{5}{2}}(\chi_{\underline{\varepsilon},\underline{\theta}})\\
		=&B^\frac{1}{5}\left(\int \chi_{\underline{\varepsilon},\underline{\theta}}(w,z) \frac{\mathds{E}(wz\sqrt{w+z}) }{wz\sqrt{w+z}}\operatorname{d}w\operatorname{d}z +O_{\underline{\varepsilon},\underline{\theta}}\left(\frac{(\log\log B)^\frac{5}{6}}{(\log B)^\frac{2-\sqrt{2}}{6}}\right)\right),
	\end{align*}
	où
	$\mathds{E}:\mathopen]0,\infty\mathclose[\to\mathopen[0,\infty\mathclose[$ est une fonction croissante en escalier.
\end{theorem}
\begin{remark*}
	La mesure obtenue fait apparaître les courbes $l_i,1\leqslant i\leqslant 3$ \eqref{eq:curvesmalldegree} de degré $2$ et la somme de leur puissances est égale exactement à la constante essentielle $\frac{5}{2}$.
\end{remark*}
\subsubsection{Déroulement du comptage}
Pour $P=[x:y]\times [s:t]$ tel que $\varrho\circ \Psi(P)\in S_1$, on a
\begin{equation}\label{eq:maxheight}
\max(|x^2 st|,| y^2 st|,| t^2 xy|,|s^2 xy|,|xyst|,|y^2t^2|)=|s^2 xy|,
\end{equation}
Pour $(a,b)\times (u,v)\in T_1$ \eqref{eq:T1}, la formule (cf. \S\ref{se:height}) pour calculer la hauteur par rapport aux paramètres $a,b,u,v$ est donnée, grâce à \eqref{eq:pgcd}, \eqref{eq:paracab} et \eqref{eq:pgcd2},  par
\begin{align*}
	\h((\varrho\circ\Psi) (a,b)\times (u,v))&=\frac{|s^2xy|}{\pgcd(x,t)\pgcd(y,s)\pgcd(y,t)}\\
	&=\frac{a^2bu^2(u-v)^3}{D_1^2D_2^2d_3^3}.
\end{align*}
La distance \eqref{eq:distance}  restreinte à $S_1$ est, d'après \eqref{eq:distance2},
\begin{equation}\label{eq:maxdistance}
d((\varrho\circ\Psi) (a,b)\times (u,v))=\frac{\frac{u^2}{v^2} -\frac{b}{a}}{\frac{b}{a}-\frac{u}{v}}.
\end{equation}

L'équivalence établie dans la Proposition \ref{le:keytranslation} nous permet de faire lien avec le problème de congruences polynomiales.
Nous allons faire une partition des paramètres $(a,b)\times (u,v)\in T_1$ selon la famille des équations $(\mathcal{E}_{C_3,D})_{C_3\in\NN_{\geqslant 1},D\in\ZZ_{\neq 0}}$ \eqref{eq:eqpellfermatgen}. D'après le Lemme \ref{le:au2-bv2}, dans $T_1$ on a $au^2-bv^2>0$, et donc $D>0$ puisque $b>a$.
Pour $D_1,D_2\in\NN_{\geqslant 1}, D_1D_2\mid D,\pgcd(D_1,D_2)=1$ et $W\in\NN_{\geqslant 1},\pgcd(W,D)=1$, nous définissons $E^{\underline{\varepsilon},\underline{\theta}}_{C_3,D,W}(\mathbf{D})$ ($\mathbf{D}$ désigne les paramètres $D_1,D_2$) l'ensemble des $(a,b)\times(u,v)\in T_1$ satisfaisant à l'équation $\mathcal{E}_{C_3,D}$ \eqref{eq:eqpellfermatgen} et vérifiant \eqref{eq:D1D2div}, \eqref{eq:Wdiv},\eqref{eq:conditionE}  et \eqref{eq:conditionEH} suivantes.
\begin{equation}\label{eq:D1D2div}
D_1=\pgcd(u^2,b),\quad D_2=\pgcd(v^2,a),
\end{equation}
\begin{equation}\label{eq:Wdiv}
 \pgcd(u-v,b-a)=\frac{b-a}{C_3W},
\end{equation}
\begin{equation}\label{eq:conditionE}
\theta_2<\frac{u}{v}\leqslant\theta_1,\quad \varepsilon_2B^{-\frac{2}{5}}<\frac{\frac{u^2}{v^2} -\frac{b}{a}}{\frac{b}{a}-\frac{u}{v} }\leqslant \varepsilon_1 B^{-\frac{2}{5}},
\end{equation}
\begin{equation}\label{eq:conditionEH}
\frac{a^2bu^2(u-v)^3}{D_1^2D_2^2d_3^3}=\frac{a^2bu^2(u-v)^3C_3^3W^3}{(b-a)^3D_1^2D_2^2}\leqslant B.
\end{equation}
D'après le Lemme \ref{le:pgcd}, \eqref{eq:Wdiv} est bien définie.
Un calcul donne que \eqref{eq:uv} est garantie par \eqref{eq:conditionE} quand $B\gg_{\overline{\varepsilon},\overline{\theta}} 1$. À l'aide du Lemme \ref{le:parametrizationofpell}, nous pouvons éliminer les paramètres $a,b$ dans \eqref{eq:conditionE} et nous obtenons
\begin{equation}\label{eq:condzoom1}
\theta_2<\frac{u}{v}\leqslant\theta_1,\quad \varepsilon_2 v(C_3uv-D)\leqslant B^\frac{2}{5}D(u+v)\leqslant \varepsilon_1 v(C_3uv-D),
\end{equation}
équivalente à \eqref{eq:conditionE}.
\begin{lemma}
On a 
\begin{equation}\label{eq:decomp1}
\begin{split}
\delta_{\varrho^{-1}(S_1)\cap U,Q,B,\frac{5}{2}}(\chi_{\underline{\varepsilon},\underline{\theta}})&=\# \left(\bigsqcup_{\substack{C_3,D,W\in\NN_{\geqslant 1},\pgcd(C_3W,D)=1\\D_1D_2\mid D,\pgcd(D_1,D_2)=1}}E^{\underline{\varepsilon},\underline{\theta}}_{C_3,D,W}(\mathbf{D})\right).
\end{split}
\end{equation}
La réunion disjointe dans \eqref{eq:decomp1} est finie.
\end{lemma}
\begin{proof}
 Le Lemme \ref{le:finitenessofpara} et les identités \eqref{eq:eqpellfermatgen} et \eqref{eq:Wdiv}
nous révèlent que 
$$\pgcd(u^2,b)\pgcd(v^2,a) c_3\ll_{\varepsilon_1} 1,$$
et
$$\frac{|au^2-bv^2|}{\pgcd(u^2,b)\pgcd(v^2,a)\pgcd(u-v,b-a)}=\frac{c_3}{C_3}\times \frac{D}{\pgcd(u^2,b)\pgcd(v^2,a)}\ll_{\varepsilon_1} 1,$$
 et donc
$C_3,D,W=O_{\varepsilon_1}(1)$.
\end{proof}
\subsubsection{Conditions de seuils}\label{se:threshold}
Avant de poursuivre le dénombrement, nous allons premièrement trouver dans cette section pour chaque point $P=(a,b)\times (u,v)\in E^{\underline{\varepsilon},\underline{\theta}}_{C_3,D,W}(\mathbf{D})$ une condition pour qu'il soit dénombré. Soit $(w,z)=\varrho\circ\Psi (P)$. Il s'avère que la condition \eqref{eq:conditionEH} donne une \emph{équation de seuil} (cf. \eqref{eq:threshold}).
\begin{lemma}\label{le:threshold0}
	Pour tout $P=(a,b)\times (u,v)\in T_1$ qui vérifie \eqref{eq:eqpellfermatgen}, \eqref{eq:D1D2div}, \eqref{eq:Wdiv} et  \eqref{eq:conditionEH}, notons
	\begin{equation}\label{eq:w0z0}
	 (w_0,z_0)=B^{\frac{2}{5}}\varrho\circ\Psi(P)=(B^{\frac{2}{5}}w,B^\frac{2}{5}z).
	\end{equation}
	Alors nous avons
\begin{equation}\label{eq:threshold}
		z_0w_0\sqrt{z_0+w_0}> \frac{D^\frac{5}{2} C_3^\frac{1}{2}W^3}{D_1^2D_2^2}.
\end{equation}
\end{lemma}
\begin{proof}
	D'après le Lemme \ref{le:parametrizationofpell}, 
	$$a=\frac{C_3 v^2+D}{k(u,v)},\quad b=\frac{C_3 u^2+D}{k(u,v)},$$
	on a $$au^2-bv^2=\frac{D(b-a)}{C_3}=\frac{D(u+v)(u-v)}{k(u,v)},$$
	$$bv-au=\frac{(C_3uv-D)(u-v)}{k(u,v)}.$$
	Nous avons donc, d'après la définition du zoom et \eqref{eq:wz},
\begin{equation}\label{eq:condzoomu1}
	z_0=B^\frac{2}{5}z=B^\frac{2}{5}\frac{au^2-bv^2}{v(bv-au)}=B^\frac{2}{5}\frac{D(u+v)}{v(C_3uv-D)}.
\end{equation}
	Nous obtenons que
\begin{equation}\label{eq:b1}
	\frac{z_0C_3 uv^2 }{D(u+v)}>B^\frac{2}{5}.
\end{equation}
	Avec la condition \eqref{eq:conditionEH} et l'identité
	$$\frac{b-a}{c_3}=\frac{u^2-v^2}{k(u,v)},$$
	on obtient
	\begin{equation}\label{eq:condheight}
	W^3(C_3v^2+D)^2(C_3u^2+D)u^2\leqslant B(u+v)^3D_1^2D_2^2.
	\end{equation}
	Cette borne supérieure nous donne la majoration
\begin{equation}\label{eq:b2}
	\frac{C_3^3W^3 u^4v^4}{D_1^2D_2^2(u+v)^3}< B.
\end{equation}
	Ces deux inégalités \eqref{eq:b1} \& \eqref{eq:b2} donnent
	$$\left(\frac{C_3^3W^3}{D_1^2D_2^2}\right)^\frac{2}{5}\frac{u^\frac{8}{5}v^\frac{8}{5}}{(u+v)^\frac{6}{5}}<\frac{z_0C_3 uv^2 }{D(u+v)},$$
	qui elle-même implique 
	$$\frac{C_3^\frac{1}{5}W^\frac{6}{5}D}{D_1^\frac{4}{5}D_2^\frac{4}{5}}< z_0\left(1+\frac{v}{u}\right)^\frac{1}{5}\left(\frac{v}{u}\right)^\frac{2}{5}.$$
	D'après la définition du paramétrage $\Psi$ et du difféomorphisme local $\varrho$ \eqref{eq:diffeomorphism}, nous avons la relation entre les paramètres $(u,v)$ et les coordonnées $(w,z)$,
\begin{equation}\label{eq:b4}
	\frac{u}{v}=\frac{z}{w}=\frac{z_0}{w_0},
\end{equation}
	d'ou nous obtenons 
	$$\frac{DC_3^\frac{1}{5}W^\frac{6}{5}}{D_1^\frac{4}{5}D_2^\frac{4}{5}}< z_0\left(1+\frac{w_0}{z_0}\right)^\frac{1}{5}\left(\frac{w_0}{z_0}\right)^\frac{2}{5}=z_0^\frac{2}{5}w_0^\frac{2}{5}(z_0+w_0)^\frac{1}{5}.$$
	En prenons la puissance $\frac{5}{2}$, nous obtenons la borne inférieure cherchée puisque $w_0,z_0>0$ dans la région $S_1$.
\end{proof}
\begin{lemma}\label{le:threshold1}
	Conservons la notation \eqref{eq:w0z0}. Pour tout $\varepsilon>0$, il existe $\mu_0>0$ ne dépendant que de $C_3,D,W,\varepsilon$ tel que pour tout $P=(a,b)\times (u,v)\in T_1$ vérifiant $z_0\leqslant \varepsilon$, \eqref{eq:eqpellfermatgen}, \eqref{eq:D1D2div}, \eqref{eq:Wdiv} et
	\begin{equation}\label{eq:w0z01}
	z_0w_0\sqrt{z_0+w_0}\geqslant\frac{D^\frac{5}{2}C_3^\frac{1}{2}W^3}{D_1^2D_2^2}+\mu_0 B^{-\frac{2}{5}},
	\end{equation}
	la condition \eqref{eq:conditionEH} soit vérifiée pour tout $B\gg_{C_3,D,W,\varepsilon} 1$.
\end{lemma}
\begin{proof}
	Notons $$\theta_0=\frac{u}{v}=\frac{z_0}{w_0}>1.$$
	La condition \eqref{eq:w0z01} implique \eqref{eq:threshold}, qui nous donne
	$$\frac{z_0^\frac{5}{2}}{\sqrt{z_0+w_0}}=\frac{\theta_0^\frac{5}{2}}{\sqrt{1+\theta_0}}\leqslant \frac{z_0^\frac{5}{2}D_1^2D_2^2}{D^\frac{5}{2}C_3^\frac{1}{2}W^3}\ll_{C_3,D,W,\varepsilon} 1,$$
	et donc 
	\begin{equation}\label{eq:b3}
	\theta_0\ll_{C_3,D,W,\varepsilon} 1.
	\end{equation}
	D'après l'identité \eqref{eq:condzoomu1}, nous avons
	$$B^{\frac{2}{5}}=\frac{z_0v(C_3uv-D)}{D(u+v)}\geqslant \frac{z_0C_3uv^2}{u+v}-\varepsilon.$$
	En combinant \eqref{eq:b3}, ceci implique aussi
	\begin{equation}\label{eq:upperboundforuv}
	u,v\ll_{C_3,D,W,\varepsilon} B^{\frac{1}{5}}.
	\end{equation}
	En utilisant l'inégalité \eqref{eq:b1} on obtient aussi
	\begin{equation}\label{eq:lowerboundforuv}
	z_0\gg_{C_3,D,W,\varepsilon} 1,\quad u,v\gg_{C_3,D,W,\varepsilon} B^{\frac{1}{5}}.
	\end{equation}
	On peut alors déduire la condition de hauteur \eqref{eq:conditionEH}, ou de la manière équivalente à \eqref{eq:condheight},
	d'une inégalité du type
	$$W^3C_3^3  u^4v^4+\mu_1u^4v^2\leqslant \left(\frac{z_0C_3uv^2}{u+v}-\varepsilon\right)^\frac{5}{2}(u+v)^3D_1^2D_2^2,$$
	où $\mu_1=O_{C_3,D,W}(1)$.
	Pour achever cette inégalité avec la condition \eqref{eq:w0z01}, il suffit d'utiliser les encadrements
	$$1<\theta_0\ll_{C_3,D,W,\varepsilon} 1,\quad 1\ll_{C_3,D,W,\varepsilon} z_0\leqslant \varepsilon,\quad u,v\gg\ll_{C_3,D,W,\varepsilon} B^{\frac{1}{5}}$$
	qui rassemblent \eqref{eq:b4}, \eqref{eq:b3}, \eqref{eq:upperboundforuv}, \eqref{eq:lowerboundforuv} et de suivre la preuve du lemme précédent.
\end{proof}

Nous concluons des Lemmes \ref{le:threshold0} et \ref{le:threshold1} qu'en prenant une fonction test \eqref{eq:chi}, nous pouvons remplacer la condition \eqref{eq:conditionEH} par \eqref{eq:threshold}.
\begin{corollary}\label{le:decomp0}
	Nous avons
	\begin{equation}\label{eq:seuilreplaceheight}
	\begin{split}
	\# E^{\underline{\varepsilon},\underline{\theta}}_{C_3,D,W}(\mathbf{D})&=\# \widetilde{E}^{\underline{\varepsilon},\underline{\theta}}_{C_3,D,W}(\mathbf{D})+O(\# \operatorname{Er}^{\underline{\varepsilon},\underline{\theta}}_{C_3,D,W}(\mathbf{D})),
	\end{split}
	\end{equation}
	où $\widetilde{E}^{\underline{\varepsilon},\underline{\theta}}_{C_3,D,W}(\mathbf{D})$ consiste en les $(a,b)\times(u,v)\in T_1$ vérifiant \eqref{eq:eqpellfermatgen}, \eqref{eq:D1D2div}, \eqref{eq:Wdiv}, \eqref{eq:condzoom1}, \eqref{eq:threshold} et $\operatorname{Er}^{\underline{\varepsilon},\underline{\theta}}_{C_3,D,W}(\mathbf{D})$ est l'ensemble des $(a,b)\times(u,v)\in T_1$ satisfaisant aux mêmes conditions précédentes sauf \eqref{eq:threshold} est remplacée par (cf. \eqref{eq:w0z01})
	\begin{equation}\label{eq:errorterm}
	\frac{D^\frac{5}{2} C_3^\frac{1}{2}W^3}{D_1^2D_2^2}\leqslant z_0w_0\sqrt{z_0+w_0}\leqslant \frac{D^\frac{5}{2} C_3^\frac{1}{2}W^3}{D_1^2D_2^2}+\mu_0 B^{-\frac{2}{5}}.
	\end{equation}
\end{corollary}

\subsubsection{Réduction du comptage}
Le but dans cette section est d'utiliser la Proposition \ref{le:keytranslation} pour réécrire le cardinal $\# \widetilde{E}^{\underline{\varepsilon},\underline{\theta}}_{C_3,D,W}(\mathbf{D})$ dans le Corollaire \ref{le:decomp0} en une somme de racines de congruences quadratiques.
Premièrement nous éliminons la condition de pgcd \eqref{eq:Wdiv}.
Au vu de la deuxième divisibilité du Lemme \ref{le:pgcd} et de \eqref{eq:W}, on a $c_3=C_3W$ et une inversion de Möbius nous donne
\begin{equation}\label{eq:decomp2}
\# \widetilde{E}^{\underline{\varepsilon},\underline{\theta}}_{C_3,D,W}(\mathbf{D})=\sum_{q\mid W} \mu\left(\frac{W}{q}\right)\# B_{C_3,D}^{\underline{\varepsilon},\underline{\theta}}\left(\mathbf{D},q\right)
\end{equation}
où pour $q, D_1,D_2\in\NN_{\geqslant 1}$ fixés, $B_{C_3,D,W}^{\underline{\varepsilon},\underline{\theta}}(\mathbf{D},q)$ est constitué des $(a,b)\times (u,v)\in T_1$ vérifiant \eqref{eq:eqpellfermatgen}, \eqref{eq:quadraticcong}, \eqref{eq:D1D2div}, \eqref{eq:condzoom1}, \eqref{eq:threshold}.

Nous allons désormais nous concentrer sur un seul ensemble $B_{C_3,D,W}^{\underline{\varepsilon},\underline{\theta}}(\mathbf{D},q)$ pour $C_3,D,W,D_1,D_2,q$ fixés.
Nous avons, grâce à la co-primalité pré-supposée, 
$$D_1=\pgcd(u^2,b)=\pgcd(u^2,au^2-bv^2)=\pgcd\left(u^2,\frac{b-a}{C_3}D\right)=\pgcd(u^2,D).$$
De la même manière, 
$$D_2=\pgcd(v^2,a)=\pgcd(v^2,D).$$
Donc nous pouvons appliquer une inversion de Möbius pour éliminer ces conditions de co-primalité. 
\begin{lemma}
	Nous avons
	\begin{equation}\label{eq:decomp3}
	\# B_{C_3,D,W}^{\underline{\varepsilon},\underline{\theta}}(\mathbf{D},q)=\sum_{\substack{h_1,h_2\in\NN_{\geqslant 1}\\D_1h_1,D_2h_2\mid D\\ \pgcd(h_1,h_2)=1}} \mu(h_1)\mu(h_2)\# \mathcal{B}_{C_3,D,W}^{\underline{\varepsilon},\underline{\theta}}(\mathbf{D},\mathbf{h},q)
	\end{equation}
	où $\mathcal{B}_{C_3,D,W}^{\underline{\varepsilon},\underline{\theta}}(\mathbf{D},\mathbf{h},q)$ est l'ensemble des $(u,v)\in \NN_{\operatorname{prem}}^2$ vérifiant \eqref{eq:quadraticcong}, \eqref{eq:threshold}, \eqref{eq:condzoom1} et
	\begin{equation}\label{eq:condxyr1}
	h_1D_1\mid u^2,\quad h_2D_2\mid v^2\quad \Longleftrightarrow \quad g(h_1D_1)\mid u,\quad g(h_2D_2)\mid v,
	\end{equation}
\end{lemma}

La condition de divisibilité \eqref{eq:quadraticcong} maintenant s'écrit, 
\begin{equation}\label{eq:condxyr2}
q(C_3u^2+D)\equiv 0 [u+v],\quad q(C_3v^2+D)\equiv 0[u+v].
\end{equation}
L'une de ces deux conditions implique l'autre. 
La restriction à la région $S_1$ implique $\frac{u}{v}>1$ et impose donc l'unicité des couples de solutions $(u,v)$.
Nous allons désormais nous concentrer sur $u$ et $u+v$.
Introduisons les notations $m,n$ de sorte que
\begin{equation}\label{eq:changeofvar1}
u=g(h_1D_1)n,\quad u+v=m.
\end{equation}
Alors la condition de co-primalité de $(u,v)$ implique celle de $(m,n)$:
\begin{equation}\label{eq:coprime1}
\pgcd(u,v)=1\Longleftrightarrow\pgcd(m,n)=1 \text{ et} \pgcd(m,g(h_1D_1))=1.
\end{equation}
Maintenant \eqref{eq:condxyr1} et \eqref{eq:condxyr2} s'écrivent
\begin{equation}\label{eq:condxys1}
m-g(h_1D_1)n\equiv 0[g(h_2D_2)],\quad q(C_3g(h_1D_1)^2 n^2+D)\equiv 0[m].
\end{equation}
Puisque $\pgcd(h_1D_1,h_2D_2)=1$, soient $1\leqslant \gamma_1<g(h_2D_2)$ tel que $$\gamma_1g(h_1D_1)\equiv 1[g(h_2D_2)].$$
Nous obtenons 
$$\gamma_1 m-\gamma_1 g(h_1D_1)n\equiv \gamma_1 m-n\equiv 0[g(h_2D_2)].$$ 
Puisque $v=m-g(h_1D_1)n\geqslant 1$, 
$$\gamma_1\leqslant \gamma_1 m-\gamma_1 g(h_1D_1)n\leqslant \gamma_1 m-n.$$
Il existe donc un entier $l\in\NN_{\geqslant 1}$ tel que
\begin{equation}\label{eq:congrln}
\gamma_1m-g(h_2D_2)l=n,
\end{equation}
et la condition de congruence dans \eqref{eq:condxys1} maintenant devient
\begin{align}\label{eq:condxys2}
q(C_3g(h_1D_1h_2D_2)^2l^2+D)\equiv 0[m]
\end{align}
avec la condition de co-primalité pour $(m,l)$ déduite de \eqref{eq:coprime1}:
\begin{equation}\label{eq:coprime2}
\pgcd(m,g(h_1D_1h_2D_2))=\pgcd(m,l)=1.
\end{equation}
Nous faisons une dernière inversion de Möbius qui élimine ces dernières conditions de pgcd.
Soient $$ e_1\mid\pgcd(m,g(h_1D_1)),\quad e_2\mid \pgcd(m,g(h_2D_2)),\quad e_3\mid \pgcd(m,l),$$ tels que $\pgcd(e_3,g(h_1D_1h_2D_2))=1$. Alors $e_1e_2e_3\mid qD$ sinon la congruence \eqref{eq:condxys2} n'a pas de solution. Écrivons 
\begin{equation}\label{eq:lprimemprime}
m=e_1e_2e_3m^\prime,\quad l=e_3l^\prime.
\end{equation}
Nous pouvons enfin réécrire \eqref{eq:condxys2} comme
\begin{equation}\label{eq:condxyss}
\frac{qC_3g(h_1D_1h_2D_2)^2e_3}{e_1e_2}(l^\prime )^2+\frac{qD}{e_1e_2e_3}\equiv 0[m^\prime].
\end{equation}
Les relations entre $(l^\prime,m^\prime)$ et $(u,v)$ sont, d'après \eqref{eq:changeofvar1}, \eqref{eq:congrln} et \eqref{eq:lprimemprime}
\begin{equation}\label{eq:xylm}
u+v=e_1e_2e_3m^\prime,\quad u= g(h_1D_1)e_3(\gamma_1e_1e_2m^\prime-g(h_2D_2) l^\prime).
\end{equation}
En résumé, nous avons établi la formule suivante.
\begin{lemma}\label{le:decomp4}
	Nous avons ($\mathbf{e}$ désigne les paramètres $e_1,e_2,e_3$)
\begin{equation}\label{eq:decomp4}
\begin{split}
	&\# \mathcal{B}_{C_3,D,W}^{\underline{\varepsilon},\underline{\theta}}(\mathbf{D},\mathbf{h},q)\\
	=&\sum_{\substack{e_1e_2\mid g(h_1D_1h_2D_2),\pgcd(e_1,e_2)=1\\e_1e_2e_3\mid qD,\pgcd(e_3,g(h_1D_1h_2D_2))=1}}\left(\prod_{j=1}^{3}\mu(\mathbf{e})\right)\# \mathcal{C}_{C_3,D}^{\underline{\varepsilon},\underline{\theta}}(\mathbf{D},\mathbf{h},\mathbf{e},q),
	\end{split}
\end{equation}
	où $\mathcal{C}_{C_3,D,W}^{\underline{\varepsilon},\underline{\theta}}(\mathbf{D},\mathbf{h},\mathbf{e},q)$ est l'ensemble des couples $(l^\prime,m^\prime)\in\NN_{\geqslant 1}\times\NN_{\geqslant 1}$ vérifiant les conditions \eqref{eq:xylm}, \eqref{eq:condzoom1}, \eqref{eq:threshold} et
\begin{equation}\label{eq:condcongtocount}
	\mathcal{F}(l^\prime)\equiv 0[m^\prime],
\end{equation}
	où 
	\begin{equation}\label{eq:formF}
	\mathcal{F}(X)=\mathcal{F}_{C_3,D,\mathbf{h}\mathbf{D},\mathbf{e},q}(X)=\frac{qC_3g(h_1D_1h_2D_2)^2e_3}{e_1e_2}X^2+\frac{qD}{e_1e_2e_3}\in\ZZ[X].
	\end{equation}
\end{lemma}

\subsubsection{Une étape clef}
Nous allons démontrer la formule asymptotique suivante en appliquant la Proposition \ref{prop:discrepancy}, puisque la condition \eqref{eq:condzoom1} ne donne pas directement une forme sommatoire souhaitée.  Les notations dans cette proposition et sa preuve sont indépendantes de celles utilisées avant. 

\begin{proposition}\label{po:centralcouting}
	Soient $A,X>0$ et $0<\vartheta_2<\vartheta_1\leqslant1$ vérifiant $A>\vartheta_1$. 
	Soit $G:\mathopen]0,\vartheta_1\mathclose]\to \RR_{>0}$ une fonction continue.
	Soit $F(Y)\in\ZZ[Y]$ un polynôme irréductible de degré $d\geqslant 2$. Rappelons $\alpha_d,\beta_d$ dans la Proposition \ref{prop:discrepancy}. Alors nous avons
	$$\sum_{\substack{l,m\in\NN,F(l)\equiv 0[m]\\m^2G(A-\frac{l}{m})\leqslant X}}\textbf{1}_{]\vartheta_2,\vartheta_1]}\left(A-\frac{l}{m}\right)=X^\frac{1}{2}\left(Z_F \int_{\vartheta_2}^{\vartheta_1}\frac{\operatorname{d}x}{\sqrt{G(x)}}+O\left(\frac{(\log\log X)^{\frac{\alpha_d}{2}}}{(\log X)^{\frac{\beta_d}{2}}}\right)\right).$$
\end{proposition}
\begin{proof}
	Fixons $\alpha=\frac{\alpha_d}{2},\beta=\frac{\beta_d}{2}$. Divisons l'intervalle $]\vartheta_2,\vartheta_1]$ en $O\left( \frac{(\log X)^\alpha}{(\log \log X)^\beta}\right)$ sous-intervalles $\{J_k\}_k$ où
	$$J_k=\left]r_k,r_{k+1}\right],\quad r_{k+1}=r_k+\frac{(\log \log X)^\beta}{(\log X)^\alpha}$$
	et définissons la fonction en escalier $H$ par
	$$H(x)=\min_{y\in J_k}G(y),\quad \text{pour }x\in J_k.$$
	Puisque $G$ est uniformément continue sur $[\vartheta_2,\vartheta_1]$, il existe $c_0>0$ une constante absolue telle que
\begin{equation*}
	0\leqslant \sup_{x\in [\vartheta_2,\vartheta_1]}\left(\sqrt{G(x)^{-1}}-\sqrt{H(x)^{-1}}\right)<c_0\frac{(\log \log X)^\beta}{(\log X)^\alpha}.
\end{equation*}
On obtient donc
\begin{equation}\label{eq:interestimate1}
\begin{split}
&\sum_k \frac{|J_k|}{\sqrt{H(r_k)}}-\int_{\vartheta_2}^{\vartheta_1}\frac{\operatorname{d}x}{\sqrt{G(x)}}\\
&=\sum_k\left(\int_{r_k}^{r_{k+1}}(\sqrt{H(r_k)^{-1}}-\sqrt{G(x)^{-1}})\operatorname{d}x\right)\\
&=O\left((\sum_k |J_k|)\frac{(\log\log X)^\beta}{(\log X)^\alpha}\right)=O\left(\frac{(\log\log X)^\beta}{(\log X)^\alpha}\right).
\end{split}
\end{equation}
	Nous avons, grâce à la Proposition \ref{prop:discrepancy} et \eqref{eq:interestimate1},
	\begin{align*}
	&\sum_{\substack{l,m\in\NN_{\geqslant 1},F(l)\equiv 0[m]\\m^2H(A-\frac{l}{m})\leqslant X}}\textbf{1}_{]\vartheta_2,\vartheta_1]}\left(A-\frac{l}{m}\right)\\
	&=	\sum_{\substack{l,m\in\NN_{\geqslant 1},F(l)\equiv 0[m]\\m^2H(A-\frac{l}{m})\leqslant X}}\sum_{k}\textbf{1}_{J_k}\left(A-\frac{l}{m}\right)\\
	&=\sum_{\substack{l,m\in\NN_{\geqslant 1},F(l)\equiv 0[m]\\m^2H(A-\frac{l}{m})\leqslant X}}\sum_{k}\textbf{1}_{A-J_k}\left(\frac{l}{m}\right)\\
	&=\sum_{k}\sum_{\substack{l,m\in\NN,F(l)\equiv 0[m]\\m\leqslant X^\frac{1}{2}/\sqrt{H(r_k)}}}\textbf{1}_{A-J_k}\left(\frac{l}{m}\right)\\
	&=Z_F X^\frac{1}{2} \left(\sum_{k}\frac{|J_k|}{\sqrt{H(r_k)}}\right)+O\left( X^\frac{1}{2} \sum_{k} \frac{(\log\log X)^{\alpha_d}}{(\log X)^{\beta_d}}\right)\\
	&=Z_F X^\frac{1}{2}\int_{\vartheta_2}^{\vartheta_1}\frac{\operatorname{d}x}{\sqrt{G(x)}}+O\left(X^\frac{1}{2}\frac{(\log\log X)^{\alpha}}{(\log X)^{\beta}}\right).
	\end{align*}
	Or nous avons aussi, en appliquant encore la Proposition \ref{prop:discrepancy} et le raisonnement ci-dessus pour la somme entre la plus grosse parenthèse ci-dessous,
	\begin{align*}
	&\sum_{\substack{l,m\in\NN_{\geqslant 1},F(l)\equiv 0[m]\\m^2G(A-\frac{l}{m})\leqslant X}}\textbf{1}_{]\vartheta_2,\vartheta_1]}\left(A-\frac{l}{m}\right)\\
	=&\sum_{\substack{l,m\in\NN_{\geqslant 1},F(l)\equiv 0[m]\\m^2H(A-\frac{l}{m})\leqslant X}}\textbf{1}_{]\vartheta_2,\vartheta_1]}\left(A-\frac{l}{m}\right)\\
	&+O\left(\sum_{\substack{X^\frac{1}{2}/(\sqrt{H\left(A-\frac{l}{m}\right)}+c_0\frac{(\log \log X)^\beta}{(\log X)^\alpha})<m\leqslant X^\frac{1}{2}/\sqrt{H\left(A-\frac{l}{m}\right)}\\F(l)\equiv 0[m]}}\textbf{1}_{A-]\vartheta_2,\vartheta_1]}\left(\frac{l}{m}\right)\right)\\
	=&Z_F X^\frac{1}{2}\int_{\vartheta_2}^{\vartheta_1}\frac{\operatorname{d}x}{\sqrt{G(x)}}+O\left(X^\frac{1}{2}\frac{(\log\log X)^{\alpha}}{(\log X)^{\beta}}\right).
	\end{align*}
	Ceci achève la preuve de la formule énoncée.
\end{proof}
\subsubsection{Dénouement}
Nous appliquons la Proposition \ref{po:centralcouting} pour estimer l'ensemble $\mathcal{C}_{C_3,D}^{\underline{\varepsilon},\underline{\theta}}(\mathbf{D},\mathbf{h},\mathbf{e},q)$ dans le Lemme \ref{le:decomp4}. 

\begin{corollary}
	Il existe une constante $\Gamma=\Gamma_{C_3,D}(\mathbf{D},\mathbf{h},\mathbf{e},q)>0$ telle que 
\begin{equation}\label{eq:decomp5}
\begin{split}
	&\# \mathcal{C}_{C_3,D,W}^{\underline{\varepsilon},\underline{\theta}}(\mathbf{D},\mathbf{h},\mathbf{e},q)\\
	=&B^{\frac{1}{5}}\Gamma Z_{\mathcal{F}}\int\int\chi_{\underline{\varepsilon},\underline{\theta}}(w,z)\frac{\mathfrak{E}_{C_3,D,\mathbf{D},W}(zw\sqrt{z+w})\operatorname{d}z\operatorname{d}w}{zw\sqrt{z+w}}+O\left(B^\frac{1}{5}\frac{(\log\log B)^\frac{5}{6}}{(\log B)^\frac{2-\sqrt{2}}{6}}\right),
	\end{split}
\end{equation}
où $$\mathfrak{E}_{C_3,D,\mathbf{D},W}(x)=\textbf{1}\left\{x:x>\frac{D^\frac{5}{2}C_3^\frac{1}{2}W^3}{D_1^2D_2^2}\right\},$$
$\chi_{\underline{\varepsilon},\underline{\theta}}$ est définie par \eqref{eq:chi}, $\mathcal{F}(X)=\mathcal{F}_{C_3,D,\mathbf{h}\mathbf{D},\mathbf{e},q}(X)$ est le polynôme \eqref{eq:formF} et la constant $Z_{\mathcal{F}}$ est définie dans la Proposition \ref{le:varrhoF}.
\end{corollary}
\begin{proof}
	Quitte à décomposer la région $R(\underline{\varepsilon},\underline{\theta})$ \eqref{eq:chi} en des petites pièces, on peut supposer pour la suite que tous les points dans $	\mathcal{C}_{C_3,D,W}^{\underline{\varepsilon},\underline{\theta}}(\mathbf{D},\mathbf{h},\mathbf{e},q)$ \eqref{eq:decomp4} vérifient la condition de seuil donnée par les équations \eqref{eq:threshold} dans les Lemmes \ref{le:threshold0} et \ref{le:threshold1} pour $B\gg_{C_3,D,\underline{\varepsilon},\underline{\theta},} 1$. 
	
	Écrivons par simplicité $$g_1=g(D_1h_1),\quad g_2=g(D_2h_2).$$
	Avec le changement de paramètres \eqref{eq:xylm}, nous avons 
	$$v=e_1e_2e_3m^\prime-u=e_3g_1g_2l^\prime-e_1e_2e_3(\gamma_1g_1-1)m^\prime,$$ et d'où
	\begin{equation}\label{eq:lprimemprime2}
	\frac{C_3}{D}\frac{uv^2}{u+v}=\frac{C_3e_3^2g_1^3g_2^3}{De_1e_2}(m^\prime)^2\left(\frac{e_1e_2\gamma_1}{g_2}-\frac{l^\prime}{m^\prime}\right)\left(\frac{e_1e_2}{g_1g_2}-\left(\frac{e_1e_2\gamma_1}{g_2}-\frac{l^\prime}{m^\prime}\right)\right)^2.
	\end{equation}
	Quitte à rajouter un terme d'erreur d'ordre grandeur $O(1)$, nous pouvons réécrire \eqref{eq:condzoom1} comme
	\begin{equation}\label{eq:lprimemprime1}
	\frac{e_1e_2\gamma_1}{g_2}-\frac{e_1e_2}{g_1g_2}\frac{\theta_1}{1+\theta_1}\leqslant\frac{l^\prime}{m^\prime}<\frac{e_1e_2\gamma_1}{g_2}-\frac{e_1e_2}{g_1g_2}\frac{\theta_2}{1+\theta_2};
	\end{equation}
	$$\varepsilon_2 (m^\prime)^2 G\left(A-\frac{l^\prime}{m^\prime}\right)< B^\frac{2}{5}\leqslant \varepsilon_1 (m^\prime)^2 G\left(A-\frac{l^\prime}{m^\prime}\right).$$
	Remarquons que $e_1e_2\mid g_1g_2$ et $\gamma_1g_1\geqslant 1$. Nous avons donc
	$$\frac{e_1e_2\gamma_1}{g(h_2D_2)}-\frac{l^\prime}{m^\prime}\subset \left]\frac{e_1e_2}{g_1g_2}\frac{\theta_2}{1+\theta_2},\frac{e_1e_2}{g_1g_2}\frac{\theta_1}{1+\theta_1}\right[\subset\mathopen]0,1\mathclose[,$$
	$$\frac{e_1e_2\gamma_1}{g_2}\geqslant \frac{e_1e_2}{g_1g_2}>\frac{e_1e_2}{g_1g_2}\frac{\theta}{1+\theta},\quad \forall \theta\geqslant 0.$$
	Nous appliquons la Proposition \ref{po:centralcouting} avec 
	$$G(x)=\frac{C_3e_3^2g(h_1D_1h_2D_2)^3}{De_1e_2}x\left(\frac{e_1e_2}{g_1g_2}-x\right)^2,\quad F(Y)=\mathcal{F}_{C_3,D,\mathbf{h}\mathbf{D},\mathbf{e},q}(Y)$$
	$$A=\frac{e_1e_2\gamma_1}{g_2},\quad \vartheta_i=\frac{e_1e_2}{g_1g_2}\frac{\theta_i}{1+\theta_i},\quad X=\frac{B^\frac{2}{5}}{\varepsilon_i},\quad (i=1,2),$$
	Il en découle que
	\begin{align*}
		&\# \mathcal{C}_{C_3,D,W}^{\underline{\varepsilon},\underline{\theta}}(\mathbf{D},\mathbf{h},\mathbf{e},q)\\
		&=\left(\sum_{\substack{l^\prime,m^\prime\in\NN,F(l^\prime)\equiv 0[m^\prime]\\(m^\prime)^2\varepsilon_2G(A-\frac{l^\prime}{m^\prime})\leqslant B^\frac{2}{5}}}-\sum_{\substack{l^\prime,m^\prime\in\NN,F(l^\prime)\equiv 0[m^\prime]\\(m^\prime)^2\varepsilon_1G(A-\frac{l^\prime}{m^\prime})\leqslant B^\frac{2}{5}}}\right)\textbf{1}_{[\vartheta_2,\vartheta_1]}\left(A-\frac{l^\prime}{m^\prime}\right)\\
		&=B^{\frac{1}{5}}\frac{(De_1e_2)^\frac{1}{2}Z_\mathcal{F}}{(C_3e_3^2g(h_1D_1h_2D_2)^3)^\frac{1}{2}}\left(\frac{1}{\sqrt{\varepsilon_2}}-\frac{1}{\sqrt{\varepsilon_1}}\right)\int_{\vartheta_1}^{\vartheta_2}\frac{\operatorname{d}x}{\sqrt{x\left(\frac{e_1e_2}{g_1g_2}-x\right)^2}}\\
		&\quad +O\left(B^\frac{1}{5}\frac{(\log\log B)^\frac{5}{6}}{(\log B)^\frac{2-\sqrt{2}}{6}}\right)\\
		&=B^{\frac{1}{5}}\frac{D^\frac{1}{2}Z_\mathcal{F}}{2C_3^\frac{1}{2}e_3g(D_1h_1D_2h_2)}\int_{\varepsilon_2}^{\varepsilon_1}\int_{\theta_2}^{\theta_1}\frac{\operatorname{d}z\operatorname{d}\theta}{z^\frac{3}{2}\sqrt{\theta(1+\theta)} }+O\left(B^\frac{1}{5}\frac{(\log\log B)^\frac{5}{6}}{(\log B)^\frac{2-\sqrt{2}}{6}}\right)\\
		&=B^{\frac{1}{5}}\frac{D^\frac{1}{2}Z_\mathcal{F}}{2C_3^\frac{1}{2}e_3g(D_1h_1D_2h_2)}\int\int_{\substack{\varepsilon_2\leqslant z\leqslant\varepsilon_1\\\theta_2<\frac{z}{w}<\theta_1}}\frac{\operatorname{d}z\operatorname{d}w}{zw\sqrt{z+w}}+O\left(B^\frac{1}{5}\frac{(\log\log B)^\frac{5}{6}}{(\log B)^\frac{2-\sqrt{2}}{6}}\right).
	\end{align*}
	Nous obtenons donc la formule énoncée avec
\begin{equation*}
 \Gamma_{C_3,D}(\mathbf{D},\mathbf{h},\mathbf{e},q)=\frac{D^\frac{1}{2}}{2C_3^\frac{1}{2}e_3g(D_1h_1D_2h_2)}.\qedhere
\end{equation*}
\end{proof}

	Il nous reste à traiter le terme d'erreur introduit dans \eqref{eq:seuilreplaceheight} du Corollaire \ref{le:decomp0}. Le même raisonnement comme ci-dessus nous donne l'estimation
\begin{equation}\label{eq:errortermcontrol}
\begin{split}
&\# \operatorname{Er}^{\underline{\varepsilon},\underline{\theta}}_{C_3,D,W}(\mathbf{D})\\&\ll_{C_3,D,W,\varepsilon_i}
B^{\frac{1}{5}}\int\int \chi\left(R(\underline{\varepsilon},\underline{\theta})\cap\Omega \right)\frac{\operatorname{d}z\operatorname{d}w}{zw\sqrt{z+w}}
+O\left(B^\frac{1}{5}\frac{(\log\log B)^\frac{5}{6}}{(\log B)^\frac{2-\sqrt{2}}{6}}\right)\\
&=O\left(B^\frac{1}{5}\frac{(\log\log B)^\frac{5}{6}}{(\log B)^\frac{2-\sqrt{2}}{6}}\right),
\end{split}
\end{equation}
où 
$$\Omega=\{(w,z):\frac{D^\frac{5}{2}C_3^\frac{1}{2}W^3}{D_1^2D_2^2}\leqslant zw\sqrt{z+w}\leqslant \frac{D^\frac{5}{2}C_3^\frac{1}{2}W^3}{D_1^2D_2^2}+\mu_0 (B^{-\frac{2}{5}})\}.$$
Rassemblons les égalités \eqref{eq:decomp1}, \eqref{eq:seuilreplaceheight}, \eqref{eq:decomp2}, \eqref{eq:decomp3}, \eqref{eq:decomp4}, \eqref{eq:decomp5} et \eqref{eq:errortermcontrol}, nous sommes prêt à démontrer le Théorème \ref{th:principalthm}.
\begin{proof}[Démonstration du Théorème \ref{th:principalthm}]
	Dans ce qui suit, pour éviter des formules superflues, on écrira le terme d'erreur sous la forme non-explicite $o(B^\frac{1}{5})$ car il sera clair d'où viennent ces contributions.
	Soient $C_3,D,W,D_1,D_2,q\in\NN$ tels que
	$$\pgcd(C_3W,D)=\pgcd(D_1,D_2)=1,\quad D_1D_2\mid D,\quad q\mid W.$$
	 En les fixant et en sommant sur les formules \eqref{eq:decomp3}, \eqref{eq:decomp4}, \eqref{eq:decomp5}, nous obtenons la formule pour le cardinal de l'ensemble $B_{C_3,D,W}^{\underline{\varepsilon},\underline{\theta}}(\mathbf{D},q)$ comme suit.
	Il existe une constante 
	\begin{align*}
	&Z_{C_3,D,W}(\mathbf{D},q)=\sum_{\substack{h_1,h_2\in\NN_{\geqslant 1}\\D_1h_1,D_2h_2\mid D\\ \pgcd(h_1,h_2)=1}} \mu(h_1)\mu(h_2)\\
	&\quad \times\sum_{\substack{e_1e_2\mid g(h_1D_1h_2D_2),\pgcd(e_1,e_2)=1\\e_1e_2e_3\mid qD,\pgcd(e_3,g(h_1D_1h_2D_2))=1}}\left(\prod_{j=1}^{3}\mu(\mathbf{e})\right)\frac{Z_\mathcal{F}D^\frac{1}{2}}{2C_3^\frac{1}{2}e_3g(D_1h_1D_2h_2)}
	\end{align*}
 telle que
	\begin{align*}
		&\# B_{C_3,D,W}^{\underline{\varepsilon},\underline{\theta}}(\mathbf{D},q)\\
		&=B^\frac{1}{5}Z_{C_3,D,W}(\mathbf{D},q)\int\int\chi_{\underline{\varepsilon},\underline{\theta}}(w,z)\frac{\mathfrak{E}_{C_3,D,\mathbf{D},W}\left(zw\sqrt{z+w}\right)}{zw\sqrt{z+w}}\operatorname{d}z\operatorname{d}w+o(B^\frac{1}{5}).
	\end{align*}
Nous concluons de \eqref{eq:decomp2} qu'en fixant $C_3,D,W,D_1,D_2$,
\begin{align*}
	&\# \widetilde{E}^{\underline{\varepsilon},\underline{\theta}}_{C_3,D,W}(\mathbf{D})=\sum_{q\mid W} \mu\left(\frac{W}{q}\right)B_{C_3,D,W}^{\underline{\varepsilon},\underline{\theta}}\left(\mathbf{D},q\right)\\
	&=B^\frac{1}{5}\int\int\chi_{\underline{\varepsilon},\underline{\theta}}(w,z)\frac{\mathbf{E}_{C_3,D,\mathbf{D},W}\left(zw\sqrt{z+w}\right)}{zw\sqrt{z+w}}\operatorname{d}z\operatorname{d}w+o(B^\frac{1}{5}),
\end{align*}
où $$\mathbf{E}_{C_3,D,\mathbf{D},W}(x)=\sum_{q\mid W}\mu\left(\frac{W}{q}\right)Z_{C_3,D,W}\left(\mathbf{D},q\right)\mathfrak{E}_{C_3,D,\mathbf{D},q}(x).$$
En reportant dans \eqref{eq:decomp1} et \eqref{eq:seuilreplaceheight}, rappelons l'estimation \eqref{eq:errortermcontrol}, nous obtenons finalement 
\begin{align*}
	\delta_{\varrho^{-1}(S_1)\cap U,Q,B,\frac{5}{2}}(\chi_{\underline{\varepsilon},\underline{\theta}})&=\sum_{\substack{C_3,D,W\in\NN_{\geqslant 1}\\\pgcd(C_3W,D)=1}}\sum_{\substack{D_1D_2\mid D\\\pgcd(D_1,D_2)=1}}\# E^{\underline{\varepsilon},\underline{\theta}}_{C_3,D,W}(\mathbf{D})\\
	&=B^\frac{1}{5}\int\int\chi_{\underline{\varepsilon},\underline{\theta}}(w,z)\frac{\mathds{E}\left(zw\sqrt{z+w}\right)}{zw\sqrt{z+w}}\operatorname{d}z\operatorname{d}w+o(B^\frac{1}{5}).
\end{align*}
où \begin{equation}\label{eq:E}
\mathds{E}(x)=\sum_{\substack{C_3,D,W\in\NN_{\geqslant 1}\\\pgcd(C_3W,D)=1}}\sum_{\substack{D_1D_2\mid D\\\pgcd(D_1,D_2)=1}}\mathbf{E}_{C_3,D,\mathbf{D},W}(x).\qedhere
\end{equation}
\end{proof} 
\subsection{Autre régions}\label{se:otherregions}
Dans les régions $S_2,S_3,S_4$ (\eqref{eq:S2}--\eqref{eq:S4}), la situation est similaire, malgré un changement mineur de hauteur et de distance. 
Nous rappelons le polynôme $\mathcal{F}(X)$ \eqref{eq:formF}. En écrivant
\begin{align*}
	&\frac{qC_3g(h_1D_1h_2D_2)^2e_3}{e_1e_2}X^2+\frac{qD}{e_1e_2e_3}\\
	=&\frac{q}{e_1e_2e_3}\left(C_3(g(h_1D_1h_2D_2)e_3X)^2+D\right)=\frac{q}{e_1e_2e_3}\mathcal{G}(X),
\end{align*}
le (non-)scindage du polynôme $C_3X^2+D$ équivaut à celui de $\mathcal{G}(X)$.
Rappelons la Proposition \ref{le:splitpoly} et l'observation dans \S\ref{se:otherregions}, les points rationnels dans la partie mince $M$  \eqref{eq:thinsetM} correspond à la famille d'équations $\mathcal{E}_{C_3,D}$ \eqref{eq:eqpellfermatgen} avec $-C_3D=\square$.
Alors en dehors de $M$, au moins un des $C_3,-D\neq \square$. Donc ces polynômes restent irréductibles et nous conduisent toujours au problème du congruence quadratique.
Nous esquissons ici le résultat et nous ne rentrerons pas dans le détail puisque la méthode et les calculs sont presque pareils. 
\subsubsection{Région $S_2$} 
Le Lemme \ref{le:au2-bv2} nous dit que pour $au^2-bv^2>0$ pour $(a,b)\times(u,v)\in T_2$. En reportant dans \eqref{eq:eqpellfermatgen}, nous avons toujours $D>0$.
Donc le résultat sous-entendu ressemble au Théorème \ref{th:principalthm}.
\begin{theorem}\label{th:S2}
	Pour toute $f\in\CBY$ (cf. \S\ref{se:zoomoper}), on a
	\begin{align*}
	&\delta_{\varrho^{-1}(S_2)\cap U,Q,B,\frac{5}{2}}(f)\\
	=&
	B^\frac{1}{5}\left(\int f(w,z)  \frac{\mathds{E}_{S_2}^\prime(-wz\sqrt{z+w})}{-wz\sqrt{z+w}}\operatorname{d}w\operatorname{d}z +O_{f}\left(\frac{(\log\log B)^\frac{5}{6}}{(\log B)^\frac{2-\sqrt{2}}{6}}\right)\right),
	\end{align*}
	où $\mathds{E}_{S_i}^\prime(\cdot)$ est une fonction en escalier définie de façon analogue à $\mathds{E}(\cdot)$ \eqref{eq:E}.
\end{theorem}

\subsubsection{Régions $S_3$ et $S_4$}

Comme expliqué précédemment, on doit retirer dans la somme \eqref{eq:decomp1} les ensembles $E^{\underline{\varepsilon},\underline{\theta}}_{C_3,D,W}(\mathbf{D})$ dont les paramètres vérifient $-C_3D=\square$. La mesure limite obtenue s'écrit de la façon suivante.
\begin{theorem}\label{th:S3}
		Pour toute $f\in\CBY$, on a pour $i=3,4$,
		\begin{align*}
		&\delta_{\varrho^{-1}(S_3)\cap U,Q,B,\frac{5}{2}}(f)\\
		=&
		B^\frac{1}{5}\left(\int f (w,z) \frac{\mathds{E}_{S_3}^\dprime(w(-z)\sqrt{z+w})}{w(-z)\sqrt{z+w}}\operatorname{d}w\operatorname{d}z +O_{f}\left(\frac{(\log\log B)^\frac{5}{6}}{(\log B)^\frac{2-\sqrt{2}}{6}}\right)\right);\\
		&\delta_{\varrho^{-1}(S_4)\cap U,Q,B,\frac{5}{2}}(f)\\
		=&B^\frac{1}{5}\left(\int f(w,z)\frac{\mathds{E}_{S_4}^\dprime(wz\sqrt{-(w+z)}) }{wz\sqrt{-(w+z)}} \operatorname{d}w\operatorname{d}z +O_{f}\left(\frac{(\log\log B)^\frac{5}{6}}{(\log B)^\frac{2-\sqrt{2}}{6}}\right)\right),
		\end{align*}
		où pour $i=3,4$, $$\mathds{E}_{S_i}^\dprime(\cdot)=\sum_{\substack{C_3,W\in\NN_{\geqslant 1},D\in\ZZ_{< 0}\\\pgcd(C_3W,D)=1,-DC_3\neq\square}}\sum_{\substack{D_1D_2\mid D\\\pgcd(D_1,D_2)=1}}\mathbf{E}^\dprime_{S_i,C_3,D,\mathbf{D},W}(\cdot)$$
		est définie de façon analogue à des quantités dans la preuve du Théorème \ref{th:principalthm}.
\end{theorem}
\subsection{Décompte de la partie mince}\label{se:countthethinpart}
Rappelons que $M$ est entièrement contenues dans la région $R_2$. Dans l'esprit de l'équidistribution globale, il n'est pas raisonnable à croire que la distribution locale autour de $Q$ soit décrite par une partie semi-algébrique qui n'est pas dense pour la topologie réelle. 
La majoration suivante améliore celle dans la Proposition \ref{po:uniformupperbound}.
\begin{lemma}\label{le:upperboundM}
	Pour toute région $R$ à support compact dans $S_3\cup S_4$, nous avons
	$$\delta_{M,Q,B,\frac{5}{2}}(\mathbf{1}_R)\ll_R B^\frac{1}{5}\log B.$$
\end{lemma}
\begin{proof}
	Par un argument comme le Lemme \ref{le:finitenessofpara}, nous n'avons qu'un nombre fini de paramètres $C_3,D$ possibles paramétrant localement les points de $M$ dans $R$. Soit $\varepsilon>0$ tel que $R\subset \mathbb{B}(0,\varepsilon)$. Une borne de type \eqref{eq:condzoom1} pour les points dans $T_3\cup T_4$ nous donne que, compte-tenu de la Proposition \ref{le:splitpoly},
	pour un certain $\gamma(C_3,D,\varepsilon)>0$, $$\delta_{M,Q,B,\frac{5}{2}}(\mathbf{1}_R)\leqslant \sum_{\substack{-C_3,D=\square\\C_3,D\ll_{\varepsilon} 1}} \sum_{n\leqslant \gamma(C_3,D,\varepsilon)B^\frac{1}{5}}\varrho_{C_3X^2+D}(n),$$
	chaque somme pour $C_3,D$ fixés contribuant à l'ordre de grandeur $O_{C_3,D}(B^{\frac{1}{5}}\log B)$, en utilisant la Proposition \ref{po:splitcong}.
	D'où la majoration énoncée.
\end{proof}

La Proposition \ref{po:splitcong} et les raisonnements dans \S\ref{se:criticalmeasure} nous suggèrent qu'on devrait avoir une minoration du type, pour tout $\varepsilon>0$ suffisamment grand,
$$\delta_{M,Q,B,\frac{5}{2}}(\chi(\varepsilon))\gg_\varepsilon B^\frac{1}{5}\log B.$$
(La minoration évidente $\gg_\varepsilon B^\frac{1}{5}$ peut être déduite du Théorème \ref{th:Pagelot}.)
Cependant, le résultat de Dartyge-Martin \cite[Theorem 1]{Dartyge-Martin} (cf. \S\ref{se:notuniformdistr}) suggère aussi que, si l'on prend une fonction trigonométrique dite $h$, alors il devrait exister $C(h)>0$ telle que
$$\delta_{M,Q,B,\frac{5}{2}}(h)\sim C(h)B^\frac{1}{5}.$$
Ces deux formules donnent une autre évidence en faveur de non-équirépartition des points dans $M$.
\appendix
\section{Congruences polynomiales et équidistribution modulo 1, d'après Erd\H{o}s et Hooley}\label{se:congHooley}
Nous considérons ici une version de congruences polynomiales à résidu fixé à la Hooley. 
Nous serons concernés par des dénombrements analogues au cardinal de l'ensemble des $(l,m)\in\NN^2$ tels que
\begin{equation}\label{eq:conditionF}
\lambda_2 B^\frac{1}{5}\leqslant m\leqslant \lambda_1 B^\frac{1}{5},\quad \tau_2 \leqslant \frac{l}{m}\leqslant \tau_1,\quad F(l)\equiv 0[m],
\end{equation}
où $0<\lambda_2<\lambda_1,0<\tau_2<\tau_1\leqslant 1$ et $F(X)\in\ZZ[X]$ est un polynôme de degré $\geqslant 2$.
Nous allons distinguer deux cas dans la discussion selon que $F(X)$ est irréductible ou non dans $\ZZ[X]$. 

Étant donné $F(X)\in\ZZ[X]$, on définit la fonction 
\begin{equation}\label{eq:rhoF}
\varrho_F(n)=\#\{1\leqslant k\leqslant n:F(k)\equiv 0[n]\}.
\end{equation}
C'est une fonction arithmétique multiplicative. 
On y associe la série de Dirichlet (définie pour $\Re(s)$ suffisamment grand)
\begin{equation}\label{eq:DFs}
D_F(s)=\sum_{n=1}^{\infty} \frac{\varrho_F(n)}{n^s}.
\end{equation}

\subsection{Cas irréductible}\label{se:irred}
On discute premièrement le cas où $F(X)$ est irréductible sur $\ZZ[X]$.
\subsubsection{Ordre Moyen}
Une observation qui remonte à Dedekind et Erd\H{o}s dit que la série $D_F(s)$ se comporte de façon similaire à la fonction zêta de Dedekind associée au corps de nombres engendré par une racine de $F(X)$. En conséquence on en déduit l'ordre moyen de $\varrho_F$ (Proposition \ref{le:varrhoF}), à l'aide d'une méthode d'analyse complexe standard.
\begin{proposition}[Erd\H{o}s, \cite{Erdos}]\label{le:varrhoF}
	Supposons que $F(X)\in\ZZ[X]$ est irréductible. Soit $\theta$ une racine algébrique de $F(X)$ et notons $K_F=\QQ(\theta)$ le corps de nombres qu'il génère. Alors nous avons
	$$D_F(s)=\zeta_{K_F}(s)\varPsi(s),$$
	où $\zeta_{K_F}(s)$ est la fonction zêta de Dedekind du corps $K_F$ et $\varPsi(s)$ est méromorphe et bornée dans le demi-plan $\Re(s)>\frac{1}{2}+\varepsilon,\forall \varepsilon>0$. Par conséquent, la série $D_F(s)$ admet un pôle simple en $s=1$.
	De plus, il existe $\lambda\in]0,1[$ dépendant du polynôme $F$ tel que
	$$\sum_{n\leqslant X} \varrho_F(n)=Z_F X+O(X^{\lambda}),$$
	où, 
	\begin{equation}\label{eq:zf}
	Z_F=\varPsi(1)\lim_{s\to 1} (s-1)\zeta_{K_F}(s).
	\end{equation}
\end{proposition}
\begin{proof}
	cf. e.g. \cite[\S5]{Greaves} ou \cite[\S7]{Daniel}
\end{proof}

\begin{remark*}
	Une formule asymptotique pour l'ordre moyen de $\varrho_F$ est également obtenue par Hooley \cite{Hooley3}, \cite{Hooley1} dans le cas quadratique. 
\end{remark*}

\subsubsection{Équirépartion modulo $1$}\label{se:uniformdistmodulo1}
Nous continuons à supposer dans cette section que le polynôme $F(X)$ est irréductible. Hooley \cite{Hooley2} a démontré le résultat suivant sur des sommes d'exponentielles.
Ce théorème a été énoncé pour les polynômes primitifs, mais la même preuve s'applique aussi à ceux qui ne sont pas forcément primitifs.

\begin{theorem}[Hooley \cite{Hooley2}, Theorem 1]\label{th:Hooley}
	Soient $X>1,h\in\NN_{\geqslant 1},d=\deg F(X)\geqslant 2$ et
	$$R(h,X)=\sum_{1\leqslant k\leqslant X}\sum_{\substack{F(v)\equiv 0 [k]\\ 1\leqslant v\leqslant k}}\operatorname{exp}\left(\frac{2\pi ih v}{k}\right).$$
	Alors
	$$R(h,X)=O_F\left(\frac{h^\frac{1}{2}X(\log\log X)^{\frac{1}{2}(d^2+1)}}{(\log X)^{\delta_d}}\right)\quad \text{où}\quad \delta_d=\frac{d-\sqrt{d}}{d!}.$$
	(N.B. Le résultat original de Hooley omet l'ordre de grandeur de $h$. Mais on le récupère facilement de sa preuve.)
\end{theorem}

Le but de cette section est de démontrer:
\begin{proposition}\label{prop:discrepancy}
	Pour tout intervalle compact $I\subseteq \RR$, nous avons
	$$\sum_{1\leqslant k\leqslant X}\sum_{\substack{v\in\ZZ\\F(v)\equiv 0 [k]}}\mathbf{1}_I\left(\frac{v}{k}\right)=Z_F |I| X+O\left(X\frac{(\log\log X)^{\alpha_d}}{(\log X)^{\beta_d}}\right),$$
	où $Z_F$ est \eqref{eq:zf} et $$\alpha_d=\frac{1}{3}(d^2+1),\quad \beta_d=\frac{1}{3}\delta_d.$$
\end{proposition}

Nous notons $(s_n)$ la suite construite en numérotant les nombres rationnels (pas nécessairement réduit) $\frac{v}{k}\in\mathopen[0,1\mathclose[$ tels que $F(v)\equiv 0[k]$ par rapport à l'ordre croissant des dénominateurs. C'est-à-dire $\frac{v_1}{k_1} \prec \frac{v_2}{k_2}$ comme éléments de $(s_n)$ si et seulement si
$$k_1\leqslant k_2,\quad \text{ou} \quad k_1=k_2 \text{ et } v_1\leqslant v_2.$$ 
Le théorème de Hooley implique que la suite $(s_n)$ est équirépartie modulo $1$, au sens de Weyl (cf. \cite[Chapter 1]{K-N}).
Nous avons besoin d'une estimation de la discrépance $D_N(s_n)$ de cette suite.
Pour la définition de \emph{la discrépance}, voir par exemple \cite[Chapter 2]{K-N}.
Les outils sont l'inégalité de Koksma-Denjoy (cf. \cite[p. 143]{K-N}) et celle de Erd\H{o}s-Turán (cf. \cite[Theorem 2.5 p. 112]{K-N}).
\begin{theorem}[Koksma-Denjoy]\label{th:koksma-denjoy}
	Soient $(x_n)$ une suite de nombres réels dans $[0,1[$ et $N\geqslant 1$. Soit $\phi$ une fonction mesurable à variation bornée définie sur $\mathopen]0,1\mathclose]$ (on note $V(\phi)$ la variation totale de $\phi$). Alors
	$$\left|\frac{1}{N}\sum_{n=1}^{N}\phi(x_n)-\int_{0}^{1}\phi\right|\leqslant V(\phi)D_N(x_n).$$
\end{theorem}

\begin{theorem}[Erd\H{o}s-Turán]\label{th:Erdosturan}
	Pour tout $m\in\NN_{\geqslant 1}$, on a
	$$D_N(x_n)=O\left(\frac{1}{m}+\sum_{h=1}^{m}\frac{1}{h}\left|\frac{1}{N}\sum_{n=1}^{N}\operatorname{exp}(2\pi i h x_n)\right|\right),$$
	où la constante implicite est absolue.
\end{theorem}

\begin{corollary}\label{co:estimationofdiscrepancy}
	Avec les notations ci-dessus, nous avons 
	$$D_N(s_n)=O_F\left(\frac{(\log\log N)^{\frac{1}{3}(d^2+1)}}{(\log N)^{\frac{2}{3}\delta_d}}\right).$$
\end{corollary}
\begin{proof}[Démonstration du corollaire]
	Fixons $N\in\NN_{\geqslant 1}$ et notons
	\begin{equation}\label{eq:ShN}
	S(h,N)=\sum_{n=1}^{N}\exp\left(2\pi i h x_n\right).
	\end{equation}
	Soit $M$ le dénominateur de $x_N$. D'après la Proposition \ref{le:varrhoF} et la relation suivante
	$$\sum_{k<M}\varrho_F(k)<N\leqslant \sum_{k\leqslant M}\varrho_F(k),$$
	il existe $C_1,C_2$ deux constantes absolues positives telles que
\begin{equation}\label{eq:MN}
	C_1 M\leqslant N\leqslant C_2M.
\end{equation}
	Nous avons aussi la comparaison suivante 
	$$S(h,N)=R(h,M)+O(\varrho_F(M))=R(h,M)+O(M^\varepsilon),\quad \forall\varepsilon>0.$$
	D'où, en utilisant le Théorème \ref{th:Hooley} et l'estimation d'Erd\H{o}s-Turán (Théorème \ref{th:Erdosturan}), nous calculons que pour tout $m\in\NN_{\geqslant 1}$,
	\begin{align*}
	D_N(s_n)&=O\left(\frac{1}{m}+\sum_{h=1}^{m}\frac{1}{h}\left|\frac{1}{N}S(h,N)\right|\right)\\
	&=O\left(\frac{1}{m}+\sum_{h=1}^{m}\frac{1}{h}\left(\left|\frac{R(h,M)}{M}\right|+\frac{1}{M^{1-\varepsilon}}\right)\right)\\
	&=O_F\left(\frac{1}{m}+\sum_{h=1}^{m}\frac{1}{h}\left(h^\frac{1}{2}\frac{(\log\log M)^{\frac{1}{2}(d^2+1)}}{(\log M)^{\delta_d}}+\frac{1}{M^{1-\varepsilon}}\right)\right)\\
	&=O_F\left(\frac{1}{m}+\sqrt{m}\frac{(\log\log M)^{\frac{1}{2}(d^2+1)}}{(\log M)^{\delta_d}}+\frac{\log m}{M^{1-\varepsilon}}\right).
	\end{align*}
	En prenant $$m=\left\lfloor\left(\frac{(\log M)^{\delta_d}}{(\log\log M)^{\frac{1}{2}(d^2+1)}}\right)^\frac{2}{3}\right\rfloor,$$
	nous obtenons, compte-tenu de \eqref{eq:MN},
	$$D_N(s_n)=O_F\left(\frac{(\log\log M)^{\frac{1}{3}(d^2+1)}}{(\log M)^{\frac{2}{3}\delta_d}}\right)=O_F\left(\frac{(\log\log N)^{\frac{1}{3}(d^2+1)}}{(\log N)^{\frac{2}{3}\delta_d}}\right),$$
	la majoration souhaitée.
\end{proof}
\begin{proof}[Démonstration de la Proposition \ref{prop:discrepancy}]
	Quitte à décomposer l'intervalle $I$ en une union de sous-intervalles disjoints, on peut supposer que $|I|\leqslant 1$ et $I\subseteq ]n_0,n_0+1]$ avec $n_0\in\ZZ$. Pour $k\in\NN_{\geqslant 1}$ fixé, tout $v\in\ZZ$ tel que $\frac{v}{k}\in I$ correspond à un unique $0\leqslant v^\prime\leqslant k-1$ tel que $\frac{v^\prime}{k}\in I-n_0$.
	Posons $$N=N(X)=\sum_{k\leqslant X}\varrho_F(k).$$
	D'après la Proposition \ref{le:varrhoF}, $N\sim Z_F X$. Alors compte tenu de l'inégalité de Koksma-Denjoy (Théorème \ref{th:koksma-denjoy}) et du Corollaire \ref{co:estimationofdiscrepancy},
	\begin{align*}
	\sum_{k\leqslant X}\sum_{\substack{v\in\ZZ\\F(v)\equiv 0 [k]}}\mathbf{1}_I\left(\frac{v}{k}\right)&=\sum_{k\leqslant X}\sum_{\substack{F(v^\prime)\equiv 0 [k]\\ 0\leqslant v^\prime\leqslant k-1}}\mathbf{1}_{I-n_0}\left(\frac{v^\prime}{k}\right)\\&=\sum_{n=1}^{N}\mathbf{1}_{I-n_0}(s_n)
	=N\int_{0}^{1}\mathbf{1}_{I-n_0}(x)\operatorname{d}x+O\left(ND_N(s_n)\right)\\
	&=N|I| +O_F\left(\frac{N(\log\log N)^{\frac{1}{3}(d^2+1)}}{(\log N)^{\frac{2}{3}\delta_d}}\right)\\
	&=Z_F |I| X +O_F\left(X\frac{(\log\log X)^{\frac{1}{3}(d^2+1)}}{(\log X)^{\frac{2}{3}\delta_d}}\right).
	\end{align*}
\end{proof}

\subsection{Cas scindés}\label{se:splitcase}

Nous nous intéresserons exclusivement dans cette section aux polynômes quadratiques \emph{scindés}, i.e. réductible sur $\ZZ[X]$. 
Fixons $a\in\NN_{\geqslant 1},c\in\ZZ_{\neq 0}$. Définissons $$F(X)=aX^2+c$$
La condition que $F(X)$ soit réductible revient à (on note $\Delta(F)$ le discriminant)
$$\Delta(F)=\square\Leftrightarrow -ac=\square.$$
\subsubsection{Ordre moyen}
La fonction donnant le nombre de congruents $\varrho_F$ ainsi que la série de Dirichlet $D_F(s)$ sont de nature différente de celles dans les cas irréductibles. En effet il s'agit ici d'un problème additif de diviseurs, qui fut considéré en premier par Ingham \cite{Ingham}.
\begin{proposition}\label{po:splitcong}
	Nous avons
	$$D_F(s)=\zeta(s)^2\varPhi(s),$$
	où $\zeta(s)$ est la fonction zêta de Riemann et $\varPhi(s)$ est holomorphe, bornée et sans zéros pour $\Re(s)>\frac{1}{2}+\varepsilon,\forall \varepsilon>0$.
	Par conséquent, la série $D_F(s)$ admet en $s=1$ un pôle d'ordre $2$ et converge absolument pour $\Re s>1$.
	De plus, il existe une constante $C_F>0$ telle que
	$$\sum_{n\leqslant X}\varrho_F(n)\sim C_F X\log X.$$
\end{proposition}
\subsubsection{Répartition uniforme}\label{se:notuniformdistr}
 Si l'on considère des sommes d'exponentielles $R(h,X)$ définies de façon similaire comme dans le Théorème \ref{th:Hooley} pour un polynôme $F$ quadratique scindé, Martin et Sitar ont démontré \cite[Theorem 1.4]{Martin-Sitar} que $R(h,X)=O(X(\log X)^{\sqrt{2}-1+\varepsilon})$. Comme conjecturé dans \cite[p. 14]{Martin-Sitar}, on ne devrait pas espérer une majoration du type $R(h,X)=o_h(X)$. Récemment Dartyge et Martin \cite[Theorem 1]{Dartyge-Martin} réussissent à établir la formule asymptotique 
 $$R(h,X)=C(F,h) X+O(X^{\frac{4}{5}+\varepsilon}),$$
 où $C(F,h)\neq 0$ dépend de $F$ et de $h$.
Une conséquence immédiate de leur résultat est que, en rappelant la somme d'exponentielle $S(h,X)$ \eqref{eq:ShN}, 
$$\frac{1}{X}S(h,X)\to C(F,h)\neq 0,\quad X\to \infty.$$
Le critère de Weyl (cf. \cite[Chapter 1, Theorem 2.1]{K-N}) nous dit que la suite $(s_n)$ formée par les racines de congruence n'est pas uniformément répartie modulo $1$. 
\section*{Remerciements}
Ce travail fait suite à la thèse de doctorat de l'auteur réalisée à l'Université Grenoble Alpes. Il tient à remercier Emmanuel Peyre de son encouragement pendant ces années et David Bourqui de son intérêt qu'il a porté à ce projet. Ses reconnaissances s'adressent à Régis de la Bretèche, Étienne Fouvry, Florent Jouve et Zhiyu Tian pour des discussions éclairantes, et également à l'arbitre anonyme pour d'utiles conseils. L'hospitalité du Max-Planck-Institut für Mathematik et le soutien de Kévin Destagnol sont sincèrement appréciés. L'auteur était partiellement supporté par le projet GARDIO, par un Riemann Fellowship, et par le budget DE1646/4-2 Deutsche Forschungsgemeinschaft.
\bibliographystyle{plain}\addcontentsline{toc}{section}{Références}
\bibliography{Y3}

\end{document}